\documentclass[12pt,english]{amsart}
\usepackage[T1]{fontenc}
\usepackage[latin9]{inputenc}
\usepackage{geometry}
\usepackage{color}
\usepackage{babel}
\usepackage{units}
\usepackage{mathtools}
\usepackage{amstext}
\usepackage{amsthm}
\usepackage{amssymb}
\usepackage{graphicx}
\usepackage[unicode=true,pdfusetitle,
 bookmarks=true,bookmarksnumbered=false,bookmarksopen=false,
 breaklinks=false,pdfborder={0 0 1},backref=false,colorlinks=false]
 {hyperref}

\usepackage{tikz-cd}

\usepackage{xpatch}
\xpatchcmd{\section}{\scshape}{\large\bfseries}{}{\PatchFailure}

\makeatletter
\numberwithin{equation}{section}
\numberwithin{figure}{section}
\theoremstyle{remark}
\newtheorem*{rem*}{\protect\remarkname}
\theoremstyle{plain}
\newtheorem{thm}{\protect\theoremname}[section]
\theoremstyle{plain}
\newtheorem{assumption}[thm]{\protect\assumptionname}
\theoremstyle{plain}
\newtheorem{cor}[thm]{\protect\corollaryname}
\theoremstyle{plain}
\newtheorem*{assumption*}{\protect\assumptionname}
\theoremstyle{definition}
\newtheorem*{defn*}{\protect\definitionname}
\theoremstyle{remark}

\theoremstyle{plain}
\newtheorem{prop}[thm]{\protect\propositionname}
\theoremstyle{plain}

\theoremstyle{plain}

\providecommand{\definitionname}{Definition}

\providecommand{\lemmaname}{Lemma}
\providecommand{\propositionname}{Proposition}
\providecommand{\remarkname}{Remark}
\providecommand{\theoremname}{Theorem}
\usepackage{upgreek}

\usepackage{enumitem}
\usepackage{latexsym}
\usepackage{manfnt}
\usepackage[sans]{dsfont}
\usepackage{palatino}
\usepackage{eulervm}
\usepackage{euscript}

\usepackage{dsfont} 

\newlength{\otstup} 
\newcommand{\uppskip}{\vspace{-2\otstup}}

\newcommand{\comment}[1]{}
\newcommand{\q}[1]{{``#1''}}

\newcommand{\matr}[1]{
\begin{pmatrix}
#1
\end{pmatrix}}

\newcommand{\smatr}[1]{\br{ \begin{smallmatrix}#1\end{smallmatrix} } }

\DeclareMathOperator{\wind}{wind}

\DeclareMathOperator{\dom}{dom}
\DeclareMathOperator{\spf}{Sf}
\DeclareMathOperator{\Mor}{Mor}
\DeclareMathOperator{\Pos}{Pos}
\DeclareMathOperator{\Spec}{Spec}
\DeclareMathOperator{\Mult}{Mult}
\renewcommand{\mod}{\;\mathrm{mod}\;}
\DeclareMathOperator{\one}{\mathds{1}} 
\DeclareMathOperator{\rank}{rank}

\newcommand{\inv}{^{-1}}
\newcommand{\txt}[1]{\;\;\text{#1}\;\;}
\newcommand{\br}[1]{\left(#1\right)}
\newcommand{\brr}[1]{\left[#1\right]}
\newcommand{\bra}[1]{\left\langle #1 \right\rangle}
\newcommand{\abs}[1]{\left|#1\right|}
\newcommand{\norm}[1]{\left\|#1\right\|}
\newcommand{\case}[1]{\begin{cases}#1\end{cases}}
\newcommand{\set}[1]{\left\{#1\right\}}
\newcommand{\sett}[2]{\left\{#1 \,:\, #2\right\}}
\newcommand{\tor}[1]{\stackrel{{#1}\hspace{0.2em}}{\longrightarrow}}
\newcommand{\darr}{\dashrightarrow} 
\newcommand\xvec[1]{\vec{\rule{0pt}{1.4ex}\smash{#1}}}
\newcommand{\fvec}{\xvec{f}}

\makeatother

\providecommand{\corollaryname}{Corollary}
\providecommand{\assumptionname}{Assumption}
\providecommand{\corollaryname}{Corollary}
\providecommand{\definitionname}{Definition}
\providecommand{\lemmaname}{Lemma}
\providecommand{\propositionname}{Proposition}
\providecommand{\remarkname}{Remark}
\providecommand{\theoremname}{Theorem}

\begin{document}

\global\long\def\rmi{\mathbf{\textrm{i}}}%
\global\long\def\rme{\mathbf{\textrm{e}}}%
\global\long\def\rmd{\mathbf{\textrm{d}}}%
\global\long\def\id{\mathbf{1}}%
\global\long\def\d{\partial}%
\global\long\def\Z{\mathbb{Z}}%
\global\long\def\R{\mathbb{R}}%
\global\long\def\C{\mathbb{C}}%
\global\long\def\N{\mathbb{N}}%
\global\long\def\Q{\mathbb{Q}}%
\global\long\def\B{\mathcal{B}}%
\global\long\def\I{\mathcal{I}}%
\global\long\def\spec#1{\mathrm{Spec}\left(#1\right)}%
\global\long\def\mult#1#2{\mathrm{Mult}_{#1}\left(#2\right)}%

\global\long\def\sf#1{\mathrm{SF}_{#1}}%


\global\long\def\RR{\mathbb{\overline{R}}}%
\global\long\def\zepi{\left[0,\pi\right)}%
\global\long\def\U{\mathcal{U}}%
\global\long\def\mi{_{\mathrm{min}}}%
\global\long\def\ma{_{\mathrm{max}}}%


\global\long\def\E{\mathcal{E}}%
\global\long\def\Ev{\mathcal{E}_{v}}%
\global\long\def\V{\mathcal{V}}%
\global\long\def\lmin{\ell\mi}%
\global\long\def\vp{v_{+}}%
\global\long\def\vm{v_{-}}%
\global\long\def\vpm{v_{\pm}}%
\global\long\def\del{\delta_{\alpha}}%


\setcounter{tocdepth}{1} 

\global\long\def\tra{\tau_{\alpha}}%
\global\long\def\tro{\tau_{0}}%
\global\long\def\trah{\hat{\tau}_{\alpha}}%
\global\long\def\troh{\hat{\tau}_{0}}%
\global\long\def\Ha{H_{\alpha}}%
\global\long\def\HBa{H^{\B}_{\alpha}}%
\global\long\def\HBo{H^{\B}_{0}}%
\global\long\def\La{\Lambda_{\alpha}}
\global\long\def\LBa{\La^{\B}}
\global\long\def\LB{\Lambda^{\B}}
\global\long\def\QQ{\mathcal{Q}}
\global\long\def\Qat{\QQ_{\alpha}^{t}}
\global\long\def\nua{\nu_{\alpha}}
\global\long\def\P{\mathcal{P}}
\global\long\def\Pa{\P_{\alpha}}
\global\long\def\Paf{\abs{\P_{\alpha}(f)}}
\global\long\def\Pof{\abs{\P_{0}(f)}}
\global\long\def\Ga{\Gamma_{\alpha}}
\global\long\def\Gaf{\Ga^f}

\global\long\def\fvp{f\thinspace(\vp)}
\global\long\def\fvm{f\thinspace(\vm)}
\global\long\def\eps{\varepsilon}
\global\long\def\oi{\left[0,\infty\right]}
\global\long\def\io{\left[-\infty,0\right]}
\global\long\def\ii{\left[-\infty,\infty\right]}


\title{Spectral flow and Robin domains on metric graphs}
\author{Ram Band, Marina Prokhorova, and Gilad Sofer}
\address{Department of Mathematics, Technion - Israel Institute of Technology (Haifa, Israel)
and Institute of Mathematics, University of Potsdam (Potsdam, Germany)}
\email{ramband@technion.ac.il}
\address{Department of Mathematics, University of Haifa (Haifa, Israel)
and Department of Mathematics, Technion - Israel Institute of Technology (Haifa, Israel)}
\email{marina.p@technion.ac.il}
\address{Department of Mathematics, Technion - Israel Institute of Technology (Haifa, Israel)}
\email{gilad.sofer@campus.technion.ac.il}

\begin{abstract}
This paper is devoted to the Neumann-Kirchhoff Laplacian on a finite metric graph. 
We prove an index theorem relating the nodal deficiency of an eigenfunction $f$
with (1) the Morse index of the Dirichlet-to-Neumann map, 
(2) its positive index and the first Betti number of the graph. 
We then generalize this result, replacing nodal points of $f$ with its \q{Robin points}
(these are points with a prescribed value of $\nicefrac{f'}{f}$,
known as the Robin parameter, or delta coupling, or cotangent of Pr\"ufer angle).
This provides the Robin count, a generalization of the nodal and Neumann counts of an eigenfunction. 
We relate the Robin count deficiency with the positive index of the Robin map 
(a generalization of the Dirichlet-to-Neumann map). 
In addition, we show that two of the relevant indices are independent of the Pr\"ufer angle.
Our main tool is the spectral flow of the Laplacian with special families of boundary conditions.
As an application of our results, we show that the spectral flow of these families 
is related to topological properties of the graph, 
such as its Betti number, the number of interaction vertices, 
and their positions with respect to the graph cycles.
\end{abstract}

\maketitle

\uppskip
\tableofcontents

\setlength{\parskip}{5pt plus 2pt minus 2pt}

\section{Introduction and main results}\label{sec: Intro and main results}

By Courant's nodal line theorem \cite{Courant23},
\begin{equation}\label{eq: def-pos}
n-\nu(f_{n})\geq0,
\end{equation}
where $f_{n}$ is the $n$-th eigenfunction of the Laplacian (or Schr\"odinger operator) 
on a manifold and $\nu(f_{n})$ is the number of nodal domains of $f_{n}$. 
The expression on the left hand side of \eqref{eq: def-pos} is often called the \emph{nodal deficiency}. 
In recent years, the nodal deficiency has gained several interesting interpretations
via index theorems, expressing it as a stability index (Morse index) of various operators. 
These theorems can be categorized into three classes, based on the relevant operator: 
(1) the two-sided Dirichlet-to-Neumann map \cite{BerCoxMar_lmp19,MR3718436}, 
(2) an energy functional on the space of partitions 
\cite{BanBerRazSmi_cmp12,BerKucSmi_gafa12,BerRazSmi_jpa12,BonHel_book17,HelHofTer_aip09,HofKenMugPlu_ieot21,KenKurLenDel_cvpde21,kennedy2023cheeger},
and (3) an eigenvalue of the magnetic Laplacian \cite{BerWey_ptrsa14,MR3125554,Col_apde13}.
Additional insights may then be obtained by drawing connections between these three types of index theorems 
\cite{BerCoxHelSun_jga22,BerCoxLatSuk_arxiv24,BerKuc_jst22,HelSun_cpde22}.
These index theorems may also be classified according to the space on which the Laplacian acts: 
a manifold, a metric graph, or a discrete graph. 
Experience shows that results obtained for one of these structures 
often lead to progress in studying the others. 
Nevertheless, there are currently two significant gaps within the theory: 
the lack of index theorems for the nodal deficiency via the magnetic Laplacians on manifolds, 
and the absence of Dirichlet-to-Neumann index theorems for the nodal deficiency on graphs. 

The current work addresses this second gap. 
Our Theorem~\ref{thm:nodal-def} provides such an index theorem for finite metric graphs. 
We then present its generalization, Theorem~\ref{thm:Robin-def}, 
counting points with a prescribed value of $\nicefrac{f'}{f}$ 
and corresponding domains (which we call \q{Robin domains}) of an eigenfunction $f$
instead of its zeros and nodal domains. 
We express the correspondent deficiency in terms of a generalization of the Dirichlet-to-Neumann map
which we call the Robin map. 

Our notion of Robin domains is also related to the recent progress
made in the study of Robin partitions \cite{kennedy2023cheeger}. The
relation between the nodal deficiency and the Robin map is provided
by the spectral flow of an appropriately defined operator family.
In addition, we show that the spectral flow of such families is related 
to the number of independent cycles in the graph 
and may be considered under the perspective of inverse spectral geometric problems
(Theorems~\ref{thm:beta-beta} and \ref{thm:SF-R-beta}). 

\subsection*{Quantum graph preliminaries.}\label{subsec:Quantum-graph-preliminaries}

A \emph{metric graph} is a triple $\Gamma=(\V,\E,\vec{\ell})$,
where $(\V,\E)$ is a combinatorial graph and 
$\vec{\ell}\in\mathbb{R}_{+}^{\E}$ is a vector of positive lengths associated with edges. 
Each $e\in\E$ is identified with the interval $\brr{0,\ell_{e}}$, 
so $\Gamma$ inherits a natural structure of a metric space.

For each vertex $v$, we denote the set of edges connected to $v$ by $\Ev$. 
The degree $\deg(v)=\abs{\Ev}$ of a vertex $v$ is the number of edges connected to $v$.
If an edge is a loop, that is, connects $v$ to itself, then both its ends are counted in $\deg(v)$.

A finite \emph{Quantum Graph} is a finite metric graph $\Gamma$ 
equipped with a self-adjoint second order differential operator $H$
and with boundary conditions at vertices which render $H$ self-adjoint.
It is common to take the Schr\"odinger operator $H=-\frac{d^{2}}{dx^{2}}+q(x)$,
where $q\in L^{\infty}(\Gamma)$ is a real-valued potential.
The domain of the corresponding unbounded operator is the subspace of the second order Sobolev space
\[ H^{2}(\Gamma)\coloneqq\oplus_{e\in\E}H^{2}(e) \]
determined by boundary conditions.
The resulting unbounded operator has compact resolvent. 
Its spectrum is an infinite discrete subset of $\R$ which is bounded from below 
and consists of eigenvalues of finite multiplicities, denoted by 
\begin{equation}\label{eq:lambdai}
	\lambda_{1}\leq\lambda_{2}\leq...\nearrow\infty.
\end{equation}
For introductory texts on quantum graphs see e.g. 
\cite{BanGnu_qg-exerices18,Berkolaiko_qg-intro17,BerKuc_graphs,GnuSmi_ap06,Kurasov_Book}.

In this paper we take $\Gamma$ to be a finite connected metric graph
equipped with the Laplacian $H=-d^{2}/dx^{2}$ (that is, without a potential), 
with the Neumann-Kirchhoff conditions at all vertices except some specified subset $\B\subset\V$. 
This condition is given by formulas:
\begin{align*}
  f \text{ is continuous at } v: \quad & f_{e}(v)=f_{e'}(v)\text{ for every } e,e'\in\Ev; \\
  \text{Current conservation: } \; & \sum_{e\in\Ev}\nabla f_{e}(v)=0, 
\end{align*}
where $f_e$ denotes the restriction of $f\in H^{2}(\Gamma)$ to the edge $e$
and $\nabla f_{e}(v)$ denotes the outgoing derivative of $f_e$ at $v$ 
(that is, the derivative taken in the direction from the vertex $v$ into the edge $e$).

\subsection*{Dirichlet-to-Neumann map.}\label{subsec:DtN-map}

Let $\B\subset\V$ be a subset of vertices and $\mu$ be a real number.
The two-sided Dirichlet-to-Neumann map (abbreviated as DtN map) $\LB(\mu)$, $\mu\in\R$,
is a self-adjoint operator on $\C^{\B}$ 
defined as follows. 
For $v\in\B$ and a vector $\xi\in\C^{\B}$, denote by $\xi(v)$ its $v$-coordinate. 
Consider the following boundary value problem on $\Gamma$: 
\begin{align*}
 & -d^{2}f/dx^{2}=\mu f  \text{ on }\Gamma,\nonumber \\
 & f_{e}(v)=\xi(v)  \text{ for }~v\in\B, \, e\in\Ev, \\ 
 & f\text{ satisfies Neumann-Kirchhoff conditions}  \text{ at vertices }v\in\V\setminus\B.
\end{align*}
Suppose that this boundary value problem has only trivial solution for $\xi=0$. 
Then it has a unique solution, which we denote by $f^{\xi}$, for every $\xi\in\C^{\B}$. 
The Dirichlet-to-Neumann map $\Lambda=\Lambda^{\B}(\mu):\C^{\B}\rightarrow\C^{\B}$ is then defined by the formula 
\begin{equation}\label{eq:DtN-def}
\br{\Lambda\xi}(v)=-\sum_{e\in\Ev}\nabla f^{\xi}_{e}(v)\text{ for }v\in\B.
\end{equation}

\begin{rem*}
We follow the definition of the two-sided Dirichlet-to-Neumann map 
presented in \cite{BerCoxHelSun_jga22,BerCoxMar_lmp19}.
We choose the minus sign in \eqref{eq:DtN-def} in order to match the standard convention 
used in higher-dimensional setting (making this map bounded from below). 
In the quantum graph community, the opposite sign is used more often, so one needs to be careful using our results.
\end{rem*}

\subsection*{Nodal deficiency.}\label{subsec:nodal-def}

Let $(\lambda,f)$ be an eigenpair of the Neumann-Kirchhoff Laplacian $H$. 
A point $x$ of $\Gamma$ is called a \emph{nodal point} of $f$ if $f(x)=0$.
The set $\P_0(f)$ of all such points is called the \emph{nodal set} of $f$
(we always suppose that it is a finite subset of $\Gamma$).
Let $\Gamma_0^f\coloneqq\overline{\Gamma\setminus\P_0(f)}$ be the graph
obtained by cutting $\Gamma$ at nodal points of $f$,
that is, the closure of $\Gamma$ with nodal points removed.
The connected components of this graph are called \emph{nodal domains} of $f$ 
and their number is denoted by $\nu_{0}(f)$. 

The natural setting for discussing nodal domains 
(or more general Robin domains that we introduce later in the paper) 
is when the function $f$ is real-valued. 
For that reason, throughout the text we choose the \emph{eigenfunctions to be real-valued} 
(this is always possible).

\begin{defn*}
We denote by $n(\lambda)$ and $N(\lambda)$ the \emph{first} and the \emph{last} place
where $\lambda$ appears in the ordered list \eqref{eq:lambdai} of eigenvalues of $H$. 
\end{defn*}

If $\lambda$ is simple, then $N(\lambda)=n(\lambda)$; in the general case 
\[ N(\lambda)=n(\lambda)+\Mult(\lambda)-1, \]
where $\Mult(\lambda)$ denotes the multiplicity of the eigenvalue $\lambda$. 
The difference $N(\lambda)-\nu_{0}(f)$ may be called the \q{upper nodal deficiency}; 
for simple eigenvalues it coincides with the usual nodal deficiency $n(\lambda)-\nu_{0}(f)$.

Consider the Dirichlet-to-Neumann map 
\[ \Lambda_0^{f}\coloneqq\Lambda^{\P_0(f)} \]
placed at the set $\P_0(f)$ of nodal points of $f$. 
Our first result expresses the nodal deficiencies of $f$ in terms of the DtN operators
$\Lambda_0^{f}(\lambda\pm\eps)\colon\C^{\B}\to\C^{\B}$,
which are well defined and invertible for $\eps>0$ small enough.

The number of negative eigenvalues, counted with multiplicities, of a self-adjoint operator $\Lambda$ 
is called the Morse index of $\Lambda$ and is denoted by $\Mor\Lambda$. 
Similarly, we denote by $\Pos\Lambda$ the positive index of $\Lambda$,
that is, the number of its positive eigenvalues. 

\begin{thm}\label{thm:nodal-def} 
Let $(\lambda,f)$ be an eigenpair
of the Neumann-Kirchhoff Laplacian $H$, with $f$ real-valued and
not vanishing at a vertex of degree larger than two. Then 
\begin{align}
N(\lambda)-\nu_{0}(f) & =\Mor\Lambda_0^{f}(\lambda+\eps), \label{eq:upper-nodal-def} \\
n(\lambda)-1 & =\Mor\Lambda_0^{f}(\lambda-\eps) \label{eq:n-1}
\end{align}
for $\eps>0$ small enough. 
The last formula can be equivalently written as 
\begin{equation}\label{eq:nodal-def}
  n(\lambda)-\nu_{0}(f)=\beta(\Gamma)-\beta(\Gamma_0^f)-\Pos\Lambda_0^{f}(\lambda-\eps),
\end{equation}
where $\Gamma_0^f$ denotes the disjoint union of nodal domains of $f$
and $\beta$ is the first Betti number (the number of independent cycles) of the corresponding graph.
\end{thm}

Formula \eqref{eq:upper-nodal-def} provides a metric graph analogue 
of the higher-dimensional index formula from \cite{BerCoxMar_lmp19,MR3718436},
and our proof follows the idea of the proof of \cite[thm.~1.1]{BerCoxMar_lmp19}.
Equality \eqref{eq:nodal-def} provides a clarification of the known bound \cite{Ber_cmp08}
\[ n(\lambda)-\nu_{0}(f)\leq\beta(\Gamma). \] 
Since \eqref{eq:nodal-def} uses the positive index of the DtN map, 
it seems that it has no analogue in higher dimension.

\begin{figure}
\includegraphics[scale=0.5]{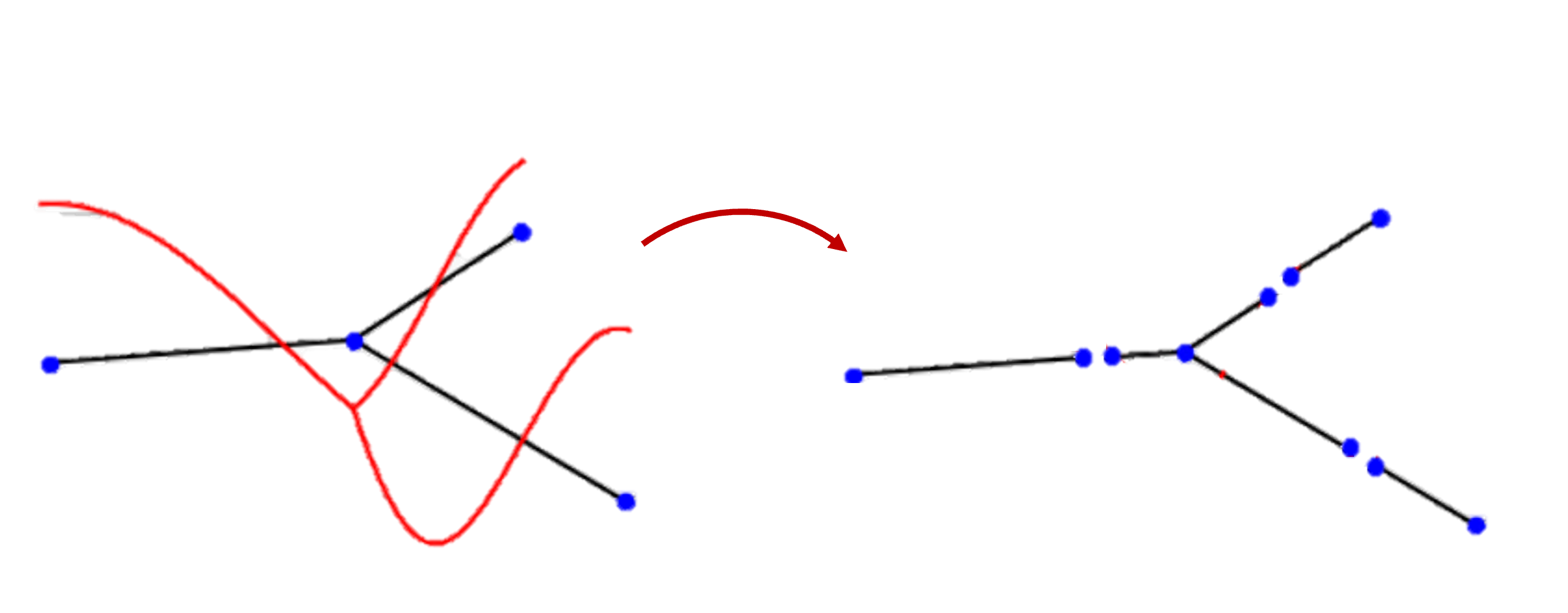}
\caption[Nodal domains.]
{\small{The eigenfunction $f_{4}$ of the Neumann-Kirchhoff Laplacian on
a star graph, partitioning it into $\nu_{0}\br{f_{4}}=4$ nodal domains.}
\label{fig:nodomains}}
\end{figure}

The non-vanishing assumption is common within the study of nodal domains 
and it is a generic one, see \cite{Alon_PhDThesis,MR4757017,BerLiu_jmaa17,Fri_ijm05}. 
We will often use a stronger assumption for an eigenpair:

\begin{assumption}\label{assu:generic}
Assume that $(\lambda,f)$ is an eigenpair of the Neumann-Kirchhoff Laplacian $H$ on $\Gamma$ 
such that $\lambda>(\pi/\lmin)^{2}$,
where $\lmin$ is the length of the shortest edge of the graph $\Gamma$,
and $f$ is real-valued and does not vanish at a vertex of degree larger than two.
\end{assumption}

Under this assumption, nodal domains have no cycles, so $\beta(\Gamma^f_0)=0$ 
and \eqref{eq:nodal-def} can be simplified.

\begin{cor}\label{cor:nodal-def-1} 
If $(\lambda,f)$ satisfies Assumption \ref{assu:generic}, then 
\begin{equation}\label{eq:nodal-def-1}
  n(\lambda)-\nu_{0}(f) = \beta(\Gamma)-\Pos\Lambda^{f}(\lambda-\eps)
\end{equation}
for $\eps>0$ small enough.
\end{cor}

\subsection*{Mixed traces.}\label{subsec:traces}

For the rest of the paper, we fix an orientation on the graph, so that each edge is directed. 
Given a function $f\in H^{2}(\Gamma)$ and a point $x$ at an edge $e\in\E$, we denote by 
\[ \tau_{0}f(x)=f(x) \txt{and} \tau_{0}'f(x)=f'(x) \]
the value and derivative of $f$ at $x$, 
where the derivative $f'$ is taken \emph{according to the orientation of $e$}. 
We generalize these two traces introducing an angle parameter
\[ \alpha\in S^{1}=\brr{0,2\pi}/\left\{0,2\pi\right\}. \] 
Namely, we define the pair of \emph{mixed, or generalized, trace maps} $\tra$ and $\tra'$ 
as the rotation of the \q{usual} pair of Dirichlet and Neumann traces:
\begin{equation}\label{eq: mixed trace map def}
\br{\begin{array}{c}\tra \\ \tra'\end{array}}
\coloneqq T_{\alpha}\cdot\br{\begin{array}{c}\tau_{0}\\ \tau_{0}'\end{array}},
\text{ where } T_{\alpha} = 
\br{\begin{array}{cc}
\cos\alpha  & -\sin\alpha \\
\sin\alpha  & \cos\alpha 
\end{array}}
\end{equation}
is the matrix of the rotation of the plain by an angle $\alpha$.
In other words, 
\[
\tra f(x)=\cos(\alpha) f(x)-\sin(\alpha)f'(x) \txt{and} 
\tra'f(x)=\sin(\alpha)f(x)+\cos(\alpha)f'(x).
\]
For $\alpha=0$ this provides the standard Dirichlet and Neumann traces.
For $\alpha=\pi/2$ we obtain the negative Neumann trace and the Dirichlet trace: 
\[ \br{\tau_{\pi/2},\tau_{\pi/2}'}=\br{-\tau_{0}',\tau_{0}}. \]

\subsection*{Vertices of degree 2.}\label{subsec:deg-2}

From now on we assume that the orientation of the graph satisfies
the following condition (such an orientation always exists).

\begin{assumption*}
Each vertex of degree two has one edge directed towards it, and the
other edge directed away from it.
\end{assumption*}

\begin{figure}
\includegraphics[scale=0.65]{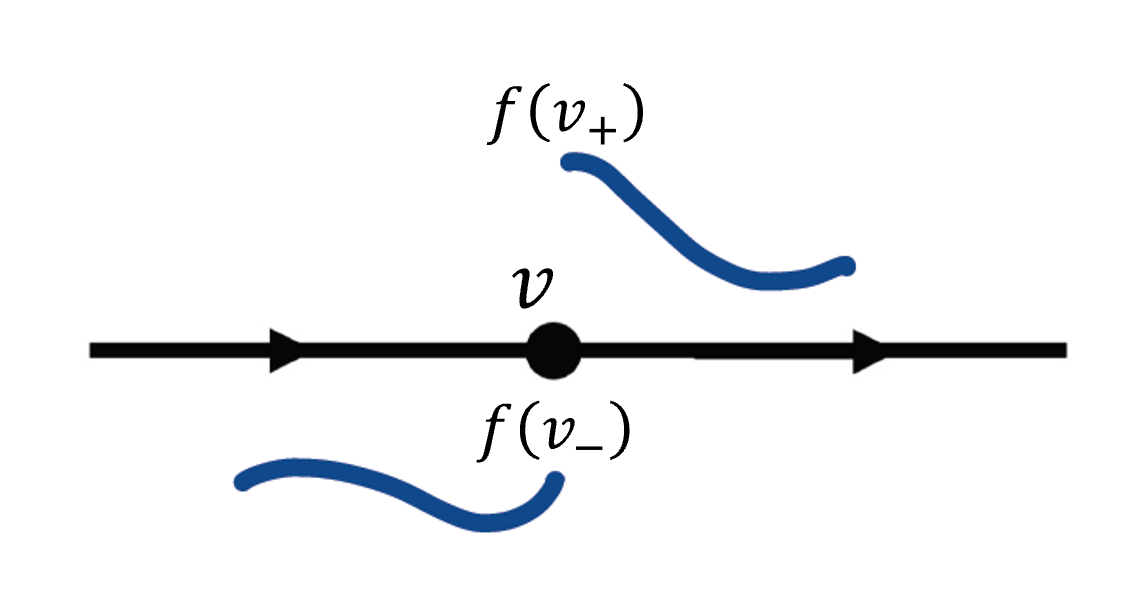}
\caption[Orientation on the edges.]
{\small{The orientation of edges at a vertex of degree two. 
The derivative $f'(\vm)$ is directed into the vertex, while $f'(\vp)$ is directed into the edge.}
\label{fig:degree-two}}
\end{figure}

Given a vertex $v$ of degree two and taking into account our assumption
for an orientation, we may refer to the values and derivatives of $f$ at each side of $v$ as $f(\vm)$, 
$f(\vp)$ and $f'(\vm)$, $f'(\vp)$, see Fig.~\ref{fig:degree-two}.
Hence, the mixed trace $\tra$ may be evaluated at both $\vm$ and $\vp$ 
and thus extends to the map $\tra^{v}\colon H^{2}(\Gamma)\to\C^{2}$,
and similarly for $\tra'$. 

Let $\B$ be a fixed subset of vertices of degree two or just internal points. 
Taking together the traces at points $v\in\B$, we obtain the mixed trace maps 
\begin{equation*}
\tra,\tra'\colon H^{2}(\Gamma)\to\C^{\B}\oplus\C^{\B}
\end{equation*}
(we do not mention $\B$ in the notation for the traces when $\B$ is fixed). 
Note that we take the same $\alpha$ parameter at all vertices of $\B$.

\subsection*{Robin points and Robin domains.}

We use mixed traces for generalizing the well known notions of nodal and Neumann points and domains.
Let $f\in H^{2}(\Gamma)$. 
We call an internal point $x$ of an edge $e\in\E$ an \emph{$\alpha$-Robin point} of $f$ if
\begin{equation*}
\tra f(x)=0, \txt{or equivalently} \sin(\alpha)f'(x)=\cos(\alpha)f(x)
\end{equation*}
(here the derivative $f'$ is taken with respect to the given orientation of $e$, as usual). 
The value $\alpha$ satisfying the formula above is known as the \emph{Prüfer angle} of $f$ at $x$.

Let $\Gamma$ be a finite graph equipped with the Neumann-Kirchhoff Laplacian $H$. 
Then vertices of degree two can be considered as internal points on longer edges; 
\emph{we extend our definition of an $\alpha$-Robin point to such \q{internal} vertices}. 
Conversely, throughout the paper, we tend to introduce additional vertices of degree two 
at $\alpha$-Robin points of a fixed eigenfunction $f$ (for a particular value of $\alpha$). 
If Neumann-Kirchhoff conditions are imposed at these additional
vertices of degree two, then there is no change in both the spectrum
and eigenfunctions of the operator. 

If $\alpha\neq0\mod\nicefrac{\pi}{2}$, then positions of Robin points,
and even their number, depend on the choice of an orientation of edges.
If one reverses the orientation of all graph edges and at the same
time replaces $\alpha$ by $\pi-\alpha$, then the corresponding Robin
points will stay the same.

Suppose that the set $\Pa(f)$ of $\alpha$-points is finite
and consider the graph $\Gaf\coloneqq\overline{\Gamma\setminus\Pa(f)}$
obtained by cutting $\Gamma$ at $\alpha$-Robin points of $f$
(that is, the closure of $\Gamma$ with $\alpha$-points removed). 

\begin{defn*}
The connected components of $\Gaf$ are called $\alpha$-\emph{Robin domains} of $f$ 
and their number is denoted by $\nua(f)$. 
The set of all $\alpha$-Robin points is denoted by $\Pa(f)$. 
\end{defn*}

An $\alpha$-Robin point is called a nodal point if $\alpha=0$ 
and a Neumann point if $\alpha=\nicefrac{\pi}{2}$. 
Similarly, the nodal and Neumann domains are Robin domains corresponding
to $\alpha=0$ and $\alpha=\nicefrac{\pi}{2}$ respectively.
See \cite{MR4314130,AloBanBerEgg_lms20} for Neumann domains on quantum graphs 
and \cite{Band2023,BanEggTay_jga20,BanFaj_ahp16,McDFul_ptrs13,Zel_sdg13} for Neumann domains in a broader context.

\subsection*{Robin map.}\label{subsec: Robin-map}

We will now define the (two-sided) \emph{$\alpha$-Robin map}, a generalization
of the two-sided Dirichlet-to-Neumann map \eqref{eq:DtN-def}. 

Let $\B\subset\V$ be a set of degree two vertices of $\Gamma$, $\alpha$ be an angle, and $\mu$ be a real number. 
The $\alpha$-Robin map $\LBa(\mu)$ is an operator on $\C^{\B}$ defined as follows. 
For every vector $\xi\in\C^{\B}$ consider the following Robin-type boundary value problem on $\Gamma$: 
\begin{align}
-d^{2}f/dx^{2}=\mu f & \quad\text{on }\Gamma,\nonumber \\
\quad\tra f(\vp)=\tra f(\vm)=\xi(v) & \quad\text{for } v\in\B, \label{eq:Robin-bdry}\\
f \text{ satisfies Neumann-Kirchhoff conditions} & \quad\text{at vertices }v\in\V\setminus\B. \nonumber 
\end{align}
Suppose that this boundary value problem has only trivial solution for $\xi=0$. 
Then it has a unique solution, which we denote by $f^{\xi}$, for every $\xi\in\C^{\B}$. 

\begin{defn*}
The $\alpha$-Robin map $\Lambda = \LBa(\mu): \C^{\B}\to\C^{\B}$ is defined by the formula 
\begin{equation}\label{eq:global-Robin}
(\Lambda\xi)(v) = \tra'f^{\xi}(\vm)-\tra'f^{\xi}(\vp) \text{ for } v\in\B.
\end{equation}
\end{defn*}

This map is self-adjoint, see Proposition \ref{prop:f-Lambda}.
For $\alpha=0$ the Robin map $\LB_{0}=\LB$ is just the usual DtN map. 
However, in contrast with the DtN map, we define the $\alpha$-Robin map 
only for a set of vertices of degree two.

Note that the $\alpha$-Robin map $\LBa$, as well as the sets of $\alpha$-Robin points and domains, 
depends only on $\alpha\mod\pi$. 
Therefore, angles $\alpha\in[0,\pi)$ are sufficient for our purposes. 
In the following, we will often \emph{omit mention of \q{modulo $\pi$}} 
while specifying values of $\alpha$, for the sake of brevity.

\begin{rem*}
1. There are other generalizations of the Dirichlet-to-Neumann map on manifolds 
to Robin boundary value problems, see e.g. \cite{MR2569392}.

2. The two-sided (\q{global}) Robin map described above can be constructed 
from similar one-sided (\q{local}) Robin maps. 
To do so, one may consider the connected components of $\Gamma\setminus\B$ 
and assign the local Robin map to each such subgraph. 
The global Robin map $\LBa(\lambda)$ is the signed sum of such local Robin maps.
For $\alpha=0$ the local Robin map is the one-sided Dirichlet-to-Neumann map, 
whereas for $\alpha=\nicefrac{\pi}{2}$ we get its inverse, 
the one-sided Neumann-to-Dirichlet map (up to a sign).
A similar approach for Dirichlet-to-Neumann maps on manifolds 
appears in \cite{BerCoxHelSun_jga22,BerCoxMar_lmp19,MR3718436}.
\end{rem*}

\subsection*{Robin deficiency formula.}\label{subsec: Robin-def}

Let $(\lambda,f)$ be an eigenpair of the Neumann-Kirchhoff Lap\-lacian $H$ and $\eps>0$ be small enough. 
Then the DtN operator $\Lambda_{0}^{f}(\lambda+\eps)$ is invertible, 
so the sum of its positive and negative indices is equal to the number $\Pof$ of nodal points. 
Hence the first part of Theorem \ref{thm:nodal-def} can be equivalently written as 
\begin{equation*}
 N(\lambda)-\nu_{0}(f)=\Pof-\Pos\br{\Lambda_{0}^{f}(\lambda+\eps)}.
\end{equation*}
Our next result generalizes this formula, replacing nodal domains
(which correspond to $\alpha=0$) by $\alpha$-Robin domains for an arbitrary value of $\alpha$. 
One may call $n(\lambda)-\nua(f)$ the \emph{$\alpha$-Robin deficiency} 
and $N(\lambda)-\nua(f)$ the \emph{upper $\alpha$-Robin deficiency} 
(for simple eigenvalues these two values coincide). Let 
\[ \La^{f}\coloneqq\La^{\Pa(f)} \]
be the $\alpha$-Robin map \eqref{eq:global-Robin} placed at the set $\B=\Pa(f)$ of $\alpha$-Robin points of $f$. 
Then $\La^{f}(\lambda+\eps)\colon\C^{\B}\to\C^{\B}$
is well defined and invertible for $|\eps|\neq0$ small enough. 

\begin{thm}\label{thm:Robin-def} 
Let $(\lambda,f)$ be an eigenpair of the Neumann-Kirchhoff Laplacian $H$ 
satisfying Assumption \ref{assu:generic} and $\alpha$ be a fixed angle. 
Then, for $\eps>0$ small enough,
\begin{equation}\label{eq:Robin-count-Pos}
N(\lambda)-\nua(f)=\Pof-\Pos\br{\La^{f}(\lambda+\eps)}.
\end{equation}
\end{thm}

\subsection*{Stability of indices.}

The parameter $\alpha$ comes into $\La^{f}=\La^{\Pa(f)}$ in two ways: 
first, it determines the set of $\alpha$-Robin points $\Pa(f)$; 
second, it provides the mixed traces which determine the Robin map at $\Pa(f)$. 
There are critical values of $\alpha$ where the number of $\alpha$-Robin points changes 
(that is, one or several of $\alpha$-points moves into a vertex of degree $\neq2$), 
which leads to the change of the size of the matrix $\La^{f}(\lambda+\eps)$.
The next theorem shows that the negative index of $\La^{f}(\lambda+\eps)$
always remains the same and only its positive index changes. 
The second part of the theorem provides a similar result for $\lambda-\eps$.

\begin{thm}\label{thm:Mor-Robin} 
Let $(\lambda,f)$ be an eigenpair of the Neumann-Kirchhoff Laplacian $H$ 
satisfying Assumption \ref{assu:generic}. 
Then, for $\eps>0$ small enough, the following indices are independent of $\alpha$:
\begin{align*}
\Mor\br{\La^{f}(\lambda+\eps)} &= \Mor\br{\Lambda_{0}^{f}(\lambda+\eps)} = N(\lambda)-\nu_{0}(f),
\\
\Pos\br{\La^{f}(\lambda-\eps)} &= \Pos\br{\Lambda_{0}^{f}(\lambda-\eps)} = \Pof-n(\lambda)+1
\end{align*}
(the notion of \q{$\eps$ small enough} may depend on $\alpha$).
\end{thm}

We prove this theorem and other results stated above in Section \ref{sec:Proof-of-index}.

\subsection*{Generalization of the delta vertex condition.}\label{subsec:delta-s}

Our proof of Theorem~\ref{thm:nodal-def} is based on the well known
loop of vertex boundary conditions, so called $\delta(t)$ vertex condition 
(it is also called Robin condition sometimes). 
This condition at a vertex $v$ is described as follows: $f$ is continuous at $v$
and the sum of outgoing derivatives of $f$ at $v$ is equal to $tf(v)$.
The \q{coupling parameter} $t$ is either a real number or infinity, that is, 
an element of the one-point compactification
\[ \RR\coloneqq\mathbb{R}\cup\infty \cong S^{1} \] 
of $\mathbb{R}$.
The value $t=\infty$ corresponds to the usual Dirichlet condition $f(v)=0$.
We now introduce a generalization of the delta condition,
which plays a key role in our proof of Theorems~\ref{thm:Robin-def} and \ref{thm:Mor-Robin}. 

\begin{defn*}
Fix an angle $\alpha$ and $t\in\RR$. Let $v\in\V$ be a degree two vertex.
The $\del(t)$ vertex condition at $v$ is given by formulas
\begin{align}
 & \tra\fvp = \tra\fvm, \label{eq:delta} \\
 & \tra'\fvp-\tra'\fvm = t\cdot\tra\fvm. \label{eq:delta'}
\end{align}
In the case $t=\infty$ these conditions are interpreted as 
\begin{equation}\label{eq: delta-inf}
\tra\fvp = \tra\fvm = 0, 
\end{equation}
that is, we cut the graph $\Gamma$ at $v$, creating two new vertices $\vm$ and $\vp$ in place of $v$, 
and impose the Robin condition $\tra f\br{\vpm}=0$ at the new vertices $\vm$ and $\vp$.
\end{defn*}

The $\del(t)$ condition depends on $\alpha\mod\pi$ 
(if $\alpha$ is replaced by $\alpha+\pi$, then both $\tra$ and $\tra'$ change sign, 
but such a simultaneous change of sign does not affect the boundary condition) 
and on the choice of orientation (since $\tra,\tra'$ depend on the edge orientation). 
For $\alpha=0$, the $\delta_{0}(t)$ condition coincides with the $\delta(t)$ condition described above. 
For $\alpha=\pi/2$, the $\delta_{\pi/2}(-t)$ condition is known as the non-symmetric
$\delta'$ (\q{delta-prime}) type condition\footnote{Since $\tau_{\pi/2}$ is the negative Neumann trace, 
we get $\delta'$ with the reversed sign of the parameter $t$.}, 
as presented in \cite{AlbeverioGesztesy_solvable,AvrExnLas_prl94,Exn_prl95}. 

For a subset $\B\subset\V$ of vertices of degree two, we define the operator $\HBa(t)$ 
as the Laplacian on $\Gamma$ equipped with the $\delta_{\alpha}(t)$ condition at every vertex $v\in\B$ 
and the Neumann-Kirchhoff conditions at all other vertices. 
For $t=0$ this provides the Neumann-Kirchhoff Laplacian $H=\HBa(0)$ regardless of the $\alpha$ value.
The family $\HBa(t)$ is continuous in the norm resolvent sense, 
see Section~\ref{subsec:Continuity-properties}, 
so we obtain a deformation of the original operator $H$ to the \q{decoupled} operator $\HBa(\infty)$.
This allows us to compare the spectral counts of these two operators.

\subsection*{Spectral flow.}\label{subsec:spectral_flow_introduction}

Another tool in our proof of theorems above is the spectral flow,  
which was first introduced by Atiyah, Patodi, and Singer in \cite{atiyah1976spectral}. 
Roughly speaking, the spectral flow counts (with signs) the number of eigenvalues passing through
zero from the start of the path to its end, see Fig.~\ref{fig:SF}. 
The spectral flow has been used for proving index theorems 
in a variety of different settings over past years, 
and was first introduced for quantum graphs in \cite{LatSuk_ams20}. 

The Schr\"odinger operator on a finite graph has compact resolvent, so every shift of it is Fredholm. 
For such operators, one can consider the spectral flow through an arbitrary spectral level $\mu\in\R$,
counting the number of eigenvalues passing through $\mu$; we will denote it $\spf_{\mu}$.
The spectral flow of a loop of operators does not depend on the choice of $\mu$: 
it counts the number of positions the spectrum is shifted by when the parameter runs the loop. 
In this case, we usually omit the subscript $\mu$ and write $\spf$ instead of $\spf_{\mu}$. 
But if the ends of a path of operators do not coincide, then the choice of $\mu$ matters.

In Section \ref{sec:SF} we recall the definition of the spectral flow and some of its basic properties. 
For a detailed exposition see \cite{booss2005unbounded,lesch2005uniqueness} 
or a recent book \cite{doll2023spectral}. 
See also the paper \cite{prokhorova_quantum} of the second author,
about the spectral flow in finite quantum graphs and some applications, 
including properties of $\HBa(t)$ families and of their generalizations.

\begin{figure}
\includegraphics[scale=0.7]{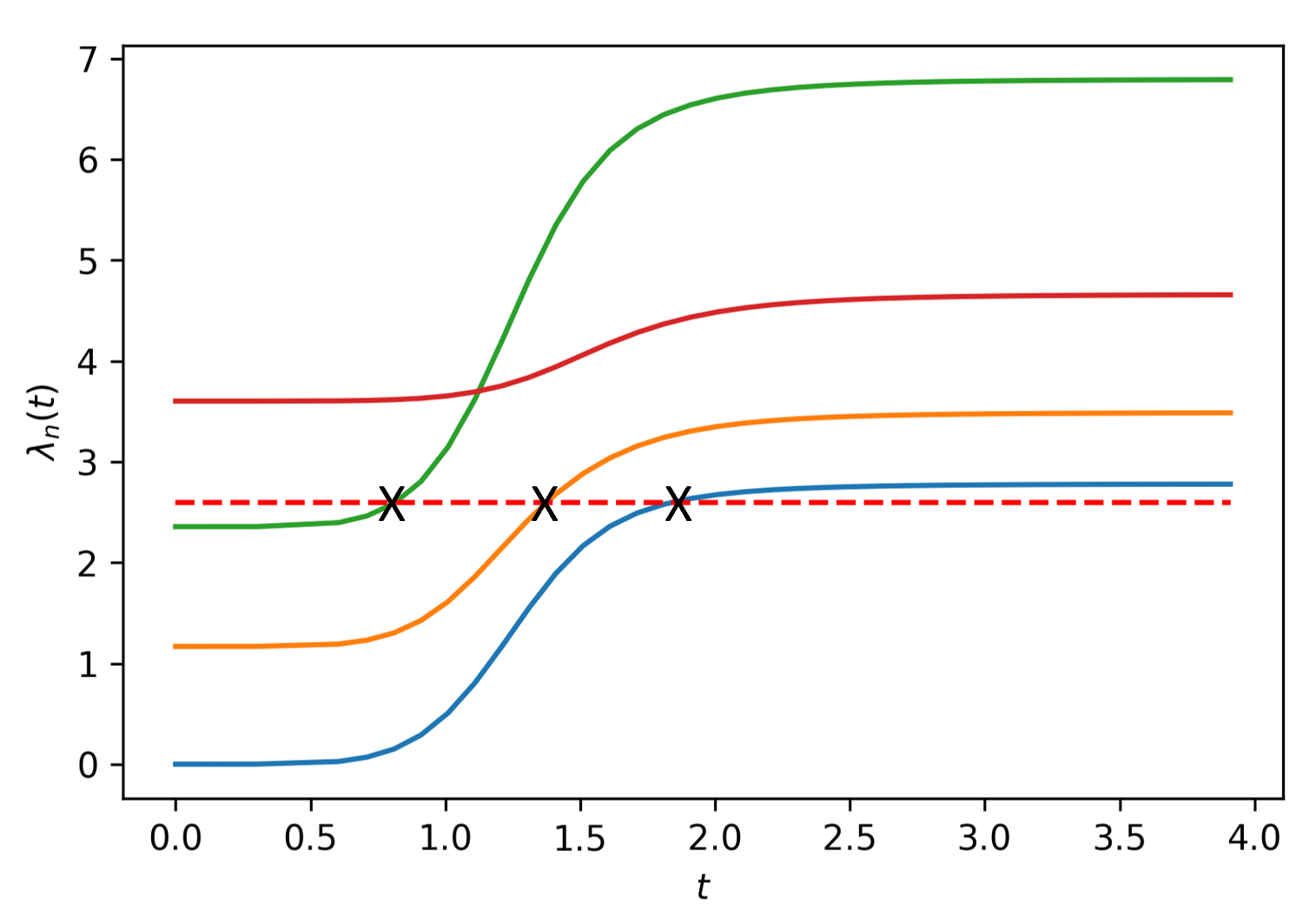}
\caption[Spectral flow.]
{\small{Illustration of the spectral flow through a given horizontal cross-section.
There are three positive intersections of the dashed line by spectral curves, 
accounting for the spectral flow through $\lambda$ equal to $3$.}
\label{fig:SF}}
\end{figure}

\subsection*{Spectral flow in terms of the Robin map.}\label{subsec:spectral_flow_Robin}

In this paper we deal only with very special families of operators, namely $\HBa(t)$, 
with $\alpha$ and $\B$ fixed and $t$ running either in an interval or in the extended real line $\RR$ 
(in the last case, the family $\HBa(t)$ is a loop).
Each spectral curve of such a family is either strictly increasing or constant, 
see Proposition~\ref{prop:monotonicity}. 
Therefore, the spectral flow $\spf_{\mu}$ of this family along any subinterval $[a,b]$ of $\RR$ 
is equal to the total number of intersections of the spectral level $\mu$ 
by non-constant spectral curves as $t$ runs $(a,b]$.
Such intersections are in one-to-one correspondence with eigenvalues of $\LBa(\mu)$,
see Proposition~\ref{prop:DTN-correspondence}.
This provides the following result,
which is one of the key ingredients in the proof of our index theorems. 

\begin{thm}\label{thm:sf-HBa-path} 
Let $H$ be the Neumann-Kirchhoff Laplacian on $\Gamma$ 
and $\B\subset\V$ be a set of degree two vertices. Then
\begin{equation}\label{eq:sf=B}
\spf\br{\HBa(t)}_{t\in\RR}=|\B|
\end{equation}
for every angle $\alpha$. 
More precisely, suppose that $\mu\in\R$ lies outside of the spectrum of both $H$ and $\HBa(\infty)$,
so that the Robin operator $\LBa(\mu)$ is well defined and invertible. Then
\begin{align}
& \spf_{\mu}\br{\HBa(t)}_{t\in\oi} = \Mor\br{\LBa(\mu)},\label{eq:sf=Mor} \\
& \spf_{\mu}\br{\HBa(t)}_{t\in\io} = \Pos\br{\LBa(\mu)}.\label{eq:sf=Pos}
\end{align}
\end{thm}

In these formulas, the values $t=\infty$ and $t=-\infty$ correspond to the same point of the circle $\RR$; 
we use the notation $\io$ in order to make clear in which path $t$ runs. 

We prove this theorem in Section \ref{sec:Robin-map}.
An alternative proof of equality \eqref{eq:sf=B} is presented in the appendix,
where we show that the spectral flow of a loop of Hamiltonians 
is equal to the winding number of the corresponding loop of boundary conditions
(Theorem \ref{thm:sf-wind-loop}) and then apply this general result to the $\HBa$-loop.

\begin{rem*}
One may deduce from Theorem~\ref{thm:sf-HBa-path} an eigenvalue bracketing (interlacing) result
similar to \cite[thm.~3.1.8]{BerKuc_graphs}. 
Impose Neumann-Kirchhoff conditions at all graph vertices except some set $\B$ of vertices of degree two. 
Then the ordered eigenvalues $\lambda_{n}(t)$ of the operators $H_{0}^{\B}(t)$ satisfy the inequality 
\begin{equation*}
\lambda_{n}(s)\leq\lambda_{n}\br{t}\leq\lambda_{n+|\B|}(s) \quad \text{for } -\infty<s<t\leq+\infty.
\end{equation*}
Similarly, if $\lambda_{n}(t)$ are the ordered eigenvalues of the operator $\HBa(t)$ 
for $\alpha\neq 0\mod\pi$, then 
\begin{equation*}
\lambda_{n}(s)\leq\lambda_{n}\br{t}\leq\lambda_{n+|\B|}(s) \quad 
  \text{for } -\infty < -s\inv < -t\inv \leq +\infty.
\end{equation*}
(Note that the special role of the point $t=\infty\in\RR=\mathbb{R}\cup\infty$ for $\alpha =0$ 
is replaced by $t=0\in\RR$ for $\alpha\neq 0$.) 
Proving such statements boils down to using Theorem~\ref{thm:sf-HBa-path} 
and monotonicity of spectral curves in order to show their spiral-like behavior
(as is nicely depicted in \cite[sec.~3.1.7, fig.~2]{BerKuc_graphs}).
In fact, the same result holds for arbitrarily fixed boundary conditions at vertices $v\notin\B$, 
and by the same reason.
\end{rem*}

\subsection*{Deformation at Robin points.} 
In our proof of Theorems~\ref{thm:nodal-def}, \ref{thm:Robin-def}, and \ref{thm:Mor-Robin} 
we place $\delta_\alpha(t)$ conditions at the set $\B=\Pa(f)$ of $\alpha$-Robin points of $f$
and leave the conditions at other points unchanged.
This provides the one-parameter family  
\[ \Ha^{f}(t)\coloneqq \Ha^{\Pa(f)}(t), \]
connecting the original operator $H=\Ha^{f}\br{0}$ with the \q{decoupled} operator $\Ha^{f}(\infty)$. 
The operator $\Ha^f(\infty)$ is the Laplacian on the disjoint union 
of $\alpha$-Robin domains of $f$,
with $\alpha$-Robin conditions at each side of $\alpha$-Robin points $v\in\Pa(f)$
and Neumann-Kirchhoff conditions at all the other (\q{internal}) vertices.
For $\alpha=0$ this corresponds to nodal domains and Dirichlet conditions at nodal points,
and for $\alpha=\nicefrac{\pi}{2}$ to Neumann domains and Neumann conditions at Neumann points.

Using the second part of Theorem~\ref{thm:sf-HBa-path},
we can reformulate Theorem \ref{thm:Mor-Robin} in terms of this family as follows.

\begin{thm}\label{thm:sf=sf} 
Let $(\lambda,f)$ be an eigenpair of $H$ satisfying  Assumption \ref{assu:generic} 
and $\eps>0$ be small enough. Then the following spectral flows are independent of $\alpha$:
\begin{align*}
& \spf_{\lambda+\eps}\br{\Ha^f(t)}_{t\in\oi} = N(\lambda)-\nu_{0}(f), 
\\
& \spf_{\lambda-\eps}\br{\Ha^f(t)}_{t\in\io} = \Pof-n(\lambda)+1.
\end{align*}
\end{thm}

\subsection*{Differences of Dirichlet, Robin, and Neumann counts}

In Section \ref{sec:sf-DRN} we describe four natural paths of operators
connecting $H_0^f(\infty)$ with $\Ha^f(\infty)$ via $H$.
Using Theorem \ref{thm:sf=sf}, we compute their spectral flows 
through levels $\lambda+\eps$ and $\lambda-\eps$
and show that two of these spectral flows are equal to the difference $\Paf-\Pof$ of the Robin and nodal counts.
For $\alpha=\nicefrac{\pi}{2}$ this provides the difference of the Neumann and nodal counts,
whose importance was pointed out in \cite{MR4314130}.

\subsection*{Spectral flow and Betti number}

Two more applications of our results, which relate the Betti number $\beta(\Gamma)$ of the graph 
with the spectral flow of appropriate operator families, are presented in Section \ref{sec:sf-beta}.
See Theorems \ref{thm:beta-beta} and \ref{thm:SF-R-beta}.

Theorem \ref{thm:beta-beta} compares the Neumann-Kirchhoff Laplacian $H=H(\Gamma)$ 
with the Neu\-mann-Kirchhoff Laplacian $H(\Gamma_\B)$ on the dissected graph $\Gamma_\B$
(the graph obtained by cutting $\Gamma$ at points of $\B$).
These two operators are connected by the family $H_{\pi/2}^{\B}(t)$, $t\in\oi$, 
and we show that the spectral flow of this family through the level $\eps$,
with $\eps>0$ small enough,
is equal to $\beta(\Gamma)-\beta(\Gamma_{\B})$,
the number of cycles removed from $\Gamma$ in the process of cutting.

The second result, Theorem \ref{thm:SF-R-beta}, concerns a simple eigenvalue $\lambda$, 
the corresponding eigenfunction $f$, and the family of Hamiltonians $\Ha^f(t)$ parametrized by $t\in\R$.
We show that the number of intersections of spectral curves with the spectral level $\lambda$
is equal to $\beta(\Gamma)$.

\subsection*{Acknowledgments}

At various stages of this work, we enjoyed fruitful discussions with
Gregory Berkolaiko, Graham Cox, Yuri Latushkin and Selim Sukhtaiev
and thank them for their excellent feedback.

All authors were supported by the Israel Science Foundation (ISF grant no.~844/19). 
R.B. and G.S. were also supported by the Binational Foundation Grant (grant no. 2016281). 
M.P. was also supported by the ISF grants no.~431/20 and 876/20 
and by the European Research Council (ERC) under the European Union's Horizon 2020 research and innovation programme (grant no.~101001677);
this work was done during her postdoctoral fellowships at the Technion -- Israel Institute of Technology 
and at the University of Haifa.

\subsection*{Note. }

The present version of this paper was finalized by the second author. 

\section{Basic properties of the $\HBa$ family}\label{sec:delta-s-properties}

Throughout this section we fix an angle $\alpha$ and a set $\B$ of degree two vertices of $\Gamma$.

\subsection*{Self-adjointness.}

In order to simplify computations, we denote the pair (column) of traces $(\tra,\tra')$ by $\trah$, 
so that 
\begin{equation}\label{eq:tra-tro}
	\trah = T_{\alpha}\troh \;\text{ and }\; \troh = T_{\alpha}^{-1}\trah = T_{-\alpha}\trah,
\end{equation}
where $T_\alpha$ is the rotation matrix \eqref{eq: mixed trace map def}.

\begin{prop}\label{prop:HBa-self-adjoint}
Each operator $\HBa(t)$ is self-adjoint.
\end{prop}

\begin{proof}
Let $f,g\in H^{2}(\Gamma)$ satisfy fixed self-adjoint boundary conditions
(say, Neumann-Kirchhoff conditions) at vertices not belonging to $\B$. 
By the Green's identity 
\[ \bra{Hf,g}_{\Gamma} - \bra{f,Hg}_{\Gamma} 
  = \sum_{v\in\B}\br{ \bra{f,g'}^{v_+}_{v_-} - \bra{f',g}^{v_+}_{v_-} }
  = \sum_{v\in\B}\bra{J\troh f, \troh g}^{v_+}_{v_-}, \]
where $J = \smatr{0 & -1 \\ 1 & 0} = T_{\pi/2}$
and $|^{v_+}_{v_-}$ denote the difference of the corresponding values at $v_+$ and $v_-$,
so that $\tra f|^{v_+}_{v_-} = \tra f(v_+)-\tra f(v_-)$ and so on.
Substituting \eqref{eq:tra-tro} in this formula 
and taking into account that $T_{-\alpha}$ is unitary and commutes with $J = T_{\pi/2}$, we obtain
\begin{align*}
	\bra{Hf,g}_{\Gamma} - \bra{f,Hg}_{\Gamma} 
		&= \sum_{v\in\B}\bra{JT_{-\alpha}\trah f, T_{-\alpha}\trah g}|^{v_+}_{v_-}
		= \sum_{v\in\B}\bra{J\trah f, \trah g}|^{v_+}_{v_-} \\
		&= \sum_{v\in\B}\br{ \bra{\tra f, \tra'g}^{v_+}_{v_-} - \bra{\tra'f,\tra g}^{v_+}_{v_-} }.
\end{align*}
If $f$ and $g$ satisfy \eqref{eq:delta} at points of $\B$, 
that is, $\tra$-traces of $f$ and $g$ are continuous at these points,
then the last expression can be simplified:
\begin{equation}\label{eq:Hfg-mixed}
	\bra{Hf,g}_{\Gamma} - \bra{f,Hg}_{\Gamma}
		= \sum_{v\in\B}\br{ \bra{\tra f(v), \tra'g|^{v_+}_{v_-}} - \bra{ \tra'f|^{v_+}_{v_-}, \tra g(v)} }.
\end{equation}
If $f$ and $g$ additionally satisfy \eqref{eq:delta'} at points of $\B$,
that is, belong to the domain of $\HBa(t)$, then the last equality takes the form
\[ \bra{Hf,g}_{\Gamma} - \bra{f,Hg}_{\Gamma}
	 = \sum_{v\in\B}\br{ \bra{\tra f(v), t\cdot\tra g(v)} - \bra{t\cdot\tra f(v), \tra g(v)} } = 0. 
\]
Hence the operator $\HBa(t)$ is symmetric. 
Since the boundary conditions have the right dimension, $\HBa(t)$ is self-adjoint. 
\end{proof}

\begin{rem*}
Our proof of this proposition is based on the observation that the 
pair of the traces $(\tra,\tra')$ is obtained from the standard pair $(\tro,\tro')$ by the rotation,
so the $\del(t)$ condition can be seen as the rotation of the standard $\delta(t)$ condition,
which is well known to be self-adjoint.
Another way to prove self-adjointness of $\HBa(t)$ is to write the $\del(t)$ condition 
in the standard form \cite[eq.~(1.4.12)]{BerKuc_graphs} 
using matrices (A.2--A.3) from \cite{Sofer2022},
and then check that the criteria of 
\cite[lem.~2.2]{KosSch_jpa99} are satisfied.
\end{rem*}

\subsection*{Lower bounds.}

It is well known that, for every individual boundary condition, a Hamiltonian on $\Gamma$ is bounded from below.
We need an analogue of this property for our families of boundary conditions.

\begin{prop}\label{prop:bdd-below}
For every angle $\alpha$ and $\eps>0$, the family $\HBa(t)$, $t\in\RR$ is uniformly bounded from below 
on the complement of the interval $(0,\eps)$ if $\alpha\neq0\mod\pi$ 
and on the complement of the interval $(-\infty,-\eps^{-1})$ if $\alpha=0\mod\pi$.
\end{prop}

\begin{proof}
Let $\Qat$ be the (closed symmetric) quadratic form associated with the self-adjoint operator $\HBa(t)$.
The lower bound of $\HBa(t)$ coincides with the lower bound of $\Qat$.
For an arbitrary boundary condition the associated quadratic form is given by 
\cite[thm.~1.4.11]{BerKuc_graphs};
its domain is contained in the first Sobolev space 
$H^1(\Gamma) = \oplus_{e\in\E}H^{1}(e)$. 
Since each $\HBa(t)$ is bounded from below, 
we can exclude any finite number of points, in particular $t=0,\infty$, from consideration. 

Consider first the case $\alpha=0$.
For $t\neq\infty$ 
\begin{equation}\label{eq:Q0}
\dom\QQ_0^t = H^{1}(\Gamma)\cap C(\Gamma), \quad\quad 
\QQ_0^t[f,f] = \norm{f'}^2_{\Gamma} + t\sum_{v\in\B}\abs{f(v)}^{2},
\end{equation}
so for $t\geq -\eps^{-1}$ the lower bound of $H^\B_0(t)$ is greater or equal 
than the lower bound of $H^\B_0(-\eps^{-1})$.

Let now $\alpha\neq 0$. 
For $t\neq 0$ the Dirichlet part of vertex boundary conditions at $v\in\B$ is trivial
and thus $\dom\Qat = H^{1}(\Gamma)\cap C\br{\Gamma\backslash\B}$ is independent of $t\neq 0$. 
Writing the formula for $\Qat$ from \cite[app.~A]{Sofer2022} in a more convenient form, 
for $t\neq0,\infty$ we get
\begin{align}\label{eq:Q1}
\Qat[f,f] = \norm{f'}^2_{\Gamma}
  + \sum_{v\in\B} \cot\alpha\br{\abs{f(v_+)}^{2} - \abs{f(\vm)}^{2}} 
  - \sum_{v\in\B} \frac{1}{t\sin^{2}\alpha}\abs{f(v_+)-f(v_-)}^{2},
\end{align}
so for $t\in(-\infty,0)\cup[\eps,\infty)$ the lower bound of $\HBa(t)$ 
is greater or equal than the lower bound of $\HBa(\eps)$.
This completes the proof of the proposition.
\end{proof}

\begin{figure}
\includegraphics[scale=0.65]{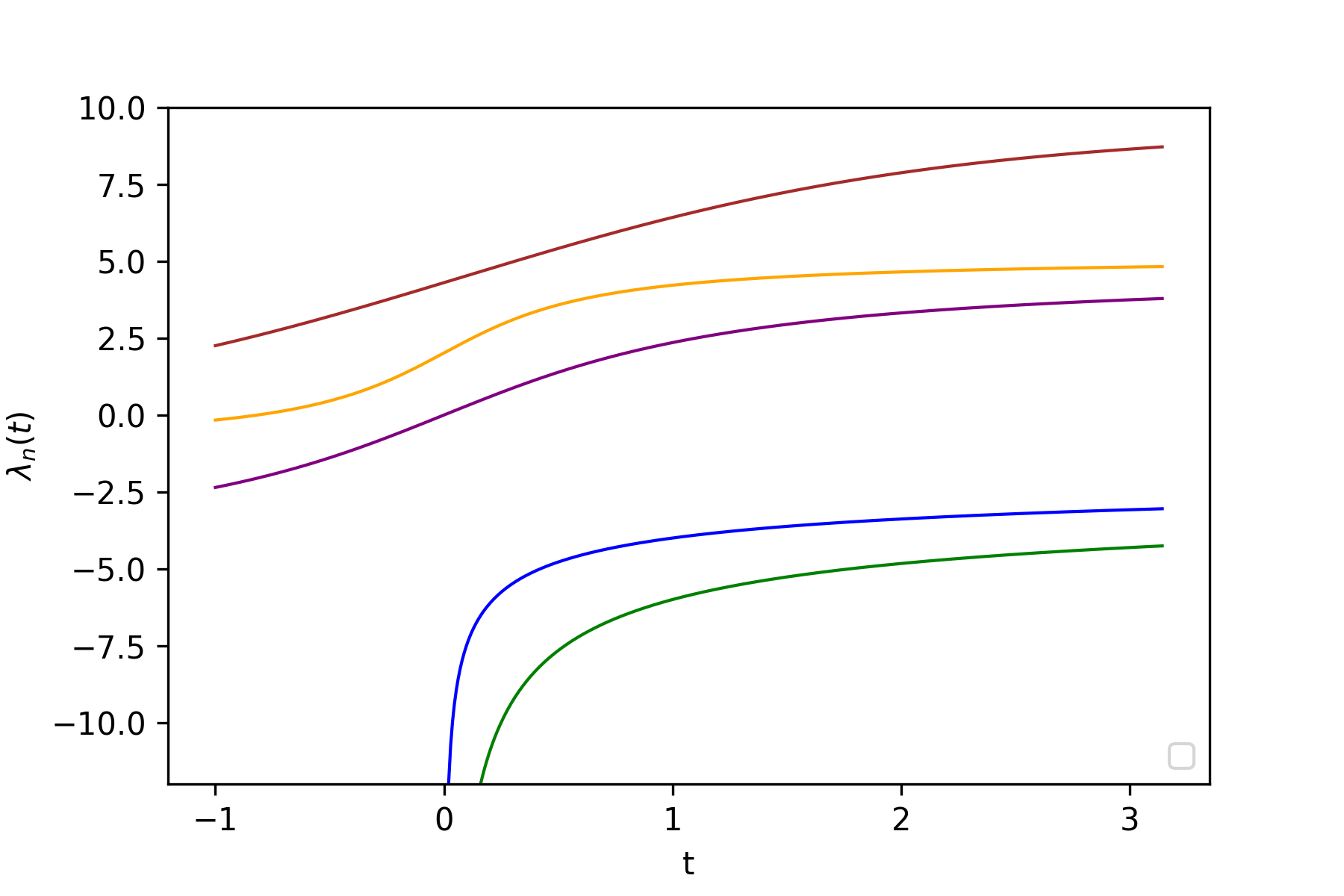}
\caption[Absence of lower bound.]
{\small{The lowest analytic eigenvalue branches of $\HBa(t)$ with $|\B| = 2$ and $\alpha\neq 0$. 
Two spectral curves tend to $-\infty$ as $t\to +0$,
showing that the family is not uniformly bounded from below.}
\label{fig:unbound}}
\end{figure}

\begin{rem*}
By the first part of Theorem \ref{thm:sf-HBa-path},
the spectral flow of the loop $\HBa(t)$, $t\in\RR$ is equal to $|\B|$.
Since it does not vanish, this family cannot be uniformly bounded from below 
on the whole extended real line $\RR$. 
Therefore, some spectral curves tend to $-\infty$.
Proposition \ref{prop:bdd-below} shows that such behavior can occur only 
at $t\to-\infty$ for $\alpha=0$ and at $t\to+0$ for $\alpha\neq0$.
The existence of such spectral curves is related to the jump of the
dimension of the Dirichlet part in a boundary condition. 
This phenomenon is discussed in more details in \cite{prokhorova_quantum}.
See also discussion in \cite{BerCoxLatSuk_arxiv24}, where the effect
of eigenvalues tending $-\infty$ is described using the Friedrichs extension.
\end{rem*}

\subsection*{Spectral curves.}

By \cite[thm.~3.1.13]{BerKuc_graphs}, the spectral curves of the family $H^v_0(t)$
(with $\delta(t)=\delta_0(t)$ condition at $v$ and fixed conditions at other vertices) 
are either constant or increasing. 
We need such a property for all our families $\HBa(t)$. 

\begin{prop}\label{prop:constant-curve}
	Let $\alpha$ be a fixed angle.
	If $f$ is an eigenfunction of $\HBa(t)$ for two different values of $t$, 
	then it is an eigenfunction of $\HBa(t)$ for every $t\in\RR$ 
	and thus gives rise to a constant (\q{horizontal}) spectral curve.
\end{prop}

\proof
If $f$ satisfies $\delta_\alpha(t)$ condition at a vertex $v\in\B$
for two different values $t_1\neq t_2\in\RR$ of $t$ simultaneously,
then $\tra\fvp = \tra\fvm = 0$ and, since at least one of $t_1,t_2$ is not equal to $\infty$, 
\eqref{eq:delta'} implies $\tra'\fvp-\tra'\fvm = 0$. 
It follows that $f$ satisfies \eqref{eq:delta'} at $v$ for all $t\in\RR$.
The boundary conditions at vertices outside of $\B$ 
are independent of $t$ and thus are also satisfied by $f$.
Therefore, $f$ belongs to the domain of $\HBa(t)$ for every $t\in\RR$.
\qed\vspace{3pt}

\begin{prop}\label{prop:curves-analytic}
The spectral set 
\begin{equation}\label{eq:sp-set}
	\sett{(t,\lambda)}{t\in\RR, \; \lambda\in\Spec\HBa(t)} 
\end{equation}
is the union of branches of real analytic curves.
\end{prop}

\begin{proof}
The boundary conditions (\ref{eq:delta}--\ref{eq:delta'}) depend analytically on $t$, 
where $t$ is considered as a point of the extended complex plane. 
By \cite[thms. 2.5.4 and 2.5.5]{BerKuc_graphs},
$\HBa(t)$ is an analytic family of unbounded operators with compact resolvents. 
The statement of the theorem then follows from \cite[thm. VII.1.8]{Kato_perturbation}.
\end{proof}


\begin{prop}\label{prop:monotonicity} 
The spectral set \eqref{eq:sp-set} is the union of strictly increasing spectral curves 
and possibly also some horizontal spectral curves $\lambda(t)\equiv\lambda_j$
described by Proposition \eqref{prop:constant-curve},
where $\set{\lambda_j}$ is a discrete set bounded from below, possibly empty.
\end{prop}

See \cite[sec.~3.1.7, fig.~2]{BerKuc_graphs} for an illustration. 
Note that \eqref{eq:sf=B} and Proposition \ref{prop:monotonicity} imply 
that the family $\HBa(t)$ has exactly $\abs{\B}$ strictly increasing spectral curves.

\begin{proof}
We follow the methods of the proof of \cite[lemma 2]{BerCoxMar_lmp19}. 
By Proposition \ref{prop:curves-analytic}, a spectral curve $\lambda(t)$ 
can be supposed to be analytic locally outside of a discrete set of exceptional values of $t$. 
Excluding a discrete set of values does not influence our conclusions, 
and we additionally exclude $t=0,\,\infty$ from consideration.
By \cite[thm.~VII.1.8]{Kato_perturbation} one can choose eigenfunctions 
$f(t)\in\ker\br{\HBa(t)-\lambda(t)}$ depending analytically on $t$.
In terms of the associated quadratic form,
\begin{equation}\label{eq:mon1}
\Qat\brr{f(t),g} = \bra{Hf(t),g} = \lambda(t)\bra{f(t),g} \text{ for every } g\in\dom\Qat.
\end{equation}
Let $\d_t$ denote the derivative with respect to $t$.
The domain of $\Qat$ does not depend on $t$ for $t\in\R\setminus 0$,
so $\dom\br{\Qat}$ contains the derivative $\d_t f$ as well.
Differentiating both sides of \eqref{eq:mon1} and then taking $g=f$, we obtain
\begin{equation}\label{eq:mon2}
\br{\d_t\Qat}\brr{f,f} + \Qat\brr{\d_t f,f} = \br{\d_t\lambda}\bra{f,f} + \lambda\bra{\d_t f,f}.
\end{equation}
On the other hand, substituting $g=\d_t f$ in \eqref{eq:mon1} provides
\begin{equation}\label{eq:mon3}
\Qat\brr{f,\d_t f} = \lambda\bra{f,\d_t f}.
\end{equation}
Taking \eqref{eq:mon2} and \eqref{eq:mon3} together, we obtain 
\begin{equation}\label{eq:mon4}
\d_t\lambda\cdot\norm{f}^2 = \br{\d_t\Qat}\brr{f,f}.
\end{equation}
Choosing a normalized eigenfunction $f=f(t)$ (that is, $\norm{f}=1$) 
and substituting (\ref{eq:Q0}, \ref{eq:Q1}) in \eqref{eq:mon4}, we see that
\[ \d_t\lambda = \case{\sum_{v\in\B}\abs{f(v)}^{2} &\text{if } \alpha=0 \\
   \br{t\sin\alpha}^{-2}\sum_{v\in\B}\abs{f(v_+)-f(v_-)}^{2}  &\text{if } \alpha\neq 0 } \]
Both expressions in this formula are non-negative, so the spectral curve $\lambda(t)$ is non-decreasing.
If $\d_t\lambda>0$ everywhere except some discrete set of points, 
then $\lambda(t)$ is a strictly increasing function.

Suppose now that $\d_t\lambda = 0$ at some non-singular point $t_0\neq0,\infty$.
For $\alpha=0$ this implies $f(t_0)\in\dom \HBa(\infty)$.
For $\alpha\neq0$ this implies continuity of $f=f(t_0)$ at points $v\in\B$;
together with continuity of the $\tra$-trace \eqref{eq:delta} this implies 
continuity of the $f'$ and $\tra'f$ at $v$, so from \eqref{eq:delta'} we get $t_0\cdot\tra f(v)=0$
and thus $f(t_0)\in\dom\HBa(\infty)$ as well.
Since $f(t_0)$ also belongs to the domain of $\HBa(t_0)$,
by Proposition \ref{prop:constant-curve} $\lambda(t)$ is a horizontal spectral curve.
The levels of such horizontal spectral curves belong to the spectrum of $H$, 
so the set consisting of these levels is discrete and bounded from below.
\end{proof}

\section{Spectral flow}\label{sec:SF}

\subsection*{Continuity properties.}\label{subsec:Continuity-properties}

In this paper we often deal with \emph{families} of operators. For
this, we need an appropriate notion of continuity. For bounded operators
we always use the norm topology. 

With a closed linear subspace $L$ of a fixed (finite- or infinite-dimensional) Hilbert space $V$ 
we associate the orthogonal projection $P_{L}\colon H\to H$ with the range $L$. 
The distance between two subspaces $L$ and $L'$ of $V$ is defined as the norm $\norm{P_{L}-P_{L'}}$. 
This provides the so called \emph{gap metric} and \emph{gap topology} on the space of closed subspaces of $V$. 
In other words, a family $L(t)\subset V$ is said to be continuous 
if the family $P(t)$ of orthogonal projections onto $L(t)$ is continuous.

With a closed operator in a Hilbert space we associate its graph. 
The gap topology on subspaces then provides the so called \emph{graph topology} on closed operators. 
In other words, a family $S(t)$ of closed operators is said to be \emph{graph
continuous} if the family $P(t)$ of orthogonal projections onto the
graphs of $S(t)$ is continuous in the norm topology. The graph topology
is the most convenient in the context of differential operators with
boundary conditions. 

If operators $S(t)$ are self-adjoint, as is the case in this paper,
then graph continuity is equivalent to continuity in the uniform resolvent
sense, that is, norm continuity of the family $t\mapsto(S(t)+\lambda)^{-1}$
for some (and then every) constant $\lambda\in\C\setminus\R$.

\subsection*{Spectral flow.}

The spectral flow is well defined for graph continuous families of self-adjoint Fredholm operators. 
Let us briefly recall its definition and properties. 
To simplify exposition, we consider only the case of operators with compact resolvents 
(we deal only with such operators in this paper). 

Let $S(t)$, $t\in[a,b]$ be a graph continuous one-parameter family 
of self-adjoint operators with compact resolvents 
(so that each operator has a discrete real spectrum) and $\mu\in\R$. 
One can always choose a partition $a=t_{0}<t_{1}<...<t_{N}=b$ 
and positive numbers $\eps_{j}$, $j=1,\ldots,N$ such that 
\begin{equation*}
\mu+\eps_{j}\notin\spec{S(t)} \txt{for} t\in\brr{t_{j-1},t_{j}}, j=1,\ldots,N.
\end{equation*}

\begin{defn*}
The spectral flow of $S(t)$ through $\mu$ is defined by the formula
\begin{equation}\label{eq:sf-definition}
\spf_{\mu}(S(t))_{t\in[a,b]}\coloneqq\sum_{j=1}^{N}
  \br{\rank \one_{[\mu,\,\mu+\eps_{j}]}\br{S(t_{j})}
	     -\rank \one_{[\mu,\,\mu+\eps_{j}]}\br{S(t_{j-1})}},
\end{equation}
where $t_{j}$ and $\eps_{j}$ are as above, 
$\one_{X}(S)$ denotes the spectral projection of the operator $S$ corresponding to a subset $X\subset\R$,
and $\rank\one_{X}$ is the rank of the projection (that is, the dimension of its range). 
This sum does not depend on the choice of a partition and numbers $\eps_{j}$. 
\end{defn*}

For $\mu\in\R$ the family $S(t)-\mu$ is graph continuous if and only if $S(t)$ is graph continuous, and  
\[ \spf_{\mu}(S(t))=\spf_{0}(S(t)-\mu). \]

\begin{rem*}
The definition above is one of two standard versions of the spectral flow. 
The alternative version is obtained by replacing closed intervals $\brr{\mu,\mu+\eps_{j}}$
in \eqref{eq:sf-definition} by semi-closed intervals $(\mu,\mu+\eps_{j}]$.
These two versions of the spectral flow differ by the number 
\[ \dim\ker\br{S(b)-\mu}-\dim\ker\br{S(a)-\mu}. \]
If these dimensions, that is, the multiplicities of $\mu$ in $S(a)$ and $S(b)$, coincide 
(which is often the case in this paper),
then the two versions of the definition provide the same value of the spectral flow. 
In particular, the spectral flow of a loop of operators is the same for both versions of the definition.
\end{rem*}

The spectral flow $\spf_{\mu}$ for graph continuous families has the following well known properties 
(in fact, it is defined by these properties and the normalization condition, see \cite{lesch2005uniqueness}):
\begin{enumerate}
\item If $\mu$ lies in the resolvent set of each $S(t)$, then $\spf_{\mu}\br{S(t)}=0$;
\item The spectral flow is additive with respect to concatenation of paths;
\item The spectral flow is homotopy invariant in the class of loops, or paths with fixed ends.
\end{enumerate}

\subsection*{Quantum graphs.}

Let $v$ be a vertex of $\Gamma$ and $f\in H^{2}(\Gamma)$. 
The boundary values $f_{e}(v)$, $e\in\E_v$ form a vector in $\C^{\E_v}$,
and similarly for the derivatives $\nabla f_{e}(v)$. 
A vertex boundary condition at $v$ for a Schr\"odinger operator $H$ is a linear
condition on these boundary values $f_{e}(v)$ and $\nabla f_{e}(v)$,
that is, it is given by a subspace $L^{v}$ of $\C^{\E_v}\oplus\C^{\E_v}$.
A (local) boundary condition for $H$ is given by a tuple of vertex boundary conditions $\br{L^{v}}_{v\in\V}$,
or equivalently by the subspace
\[ L=\bigoplus_{v\in\V}{L^{v}} \subset \C^{2d}, \quad d=\sum_{v\in\V} \deg(v) = 2\abs{\E}. \] 
We will denote the corresponding unbounded operator on $L^{2}(\Gamma)$ by $H_{L}$. 
The subspaces $L^{v}$ should satisfy additional requirements which guarantee the self-adjointness
of the corresponding operator $H_{L}$, e.g. the dimension condition $\dim L^{v}=\abs{\E_{v}}=\deg(v)$.

\begin{rem*}
 More generally, every subspace $L$ of $\C^{2d}$ determines a global boundary condition for $H$;
we will use such boundary conditions in the proof of Theorem \ref{thm:sf-wind-loop} in the appendix.
Similarly to the local case, $L$ should satisfy additional requirements 
in order to guarantee self-adjointness of $H_L$.
\end{rem*}

\begin{prop}
\label{prop:HL-cont} 
The operator $H_{L}$ depends graph continuously on the boundary condition $L\subset\C^{2d}$. 
\end{prop}

\begin{proof}
Let $H\ma$ and $H\mi=H\ma^*$ be the maximal and minimal operators 
corresponding to the differential expression $H$,
with $\dom(H\ma) = H^2(\Gamma)$ and $\dom(H\mi)$ being the set of functions $f\in H^2(\Gamma)$ 
vanishing on the ends of all edges together with their first derivatives.
Let $G\ma$ and $G\mi$ be the graphs of $H\ma$ and $H\mi$ correspondingly; 
they are closed subspaces of the Hilbert space $L^2(\Gamma)\oplus L^2(\Gamma)$.
The quotient $\C^{2d} = \dom(H\ma)/\dom(H\mi)$ can be identified, as a vector space, 
with the quotient $G\ma/G\mi$.

The graph $G_L$ of $H_L$ is a closed subspace of $G\ma$;
it is obtained as the inverse image of $L\subset\C^{2d}$ under the trace map $G\ma\to G\ma/G\mi = \C^{2d}$.
Since the trace map is surjective, $G_L$ depends continuously of $L$ as a closed subspace of $G\ma$
and thus also as a closed subspace of $L^2(\Gamma)\oplus L^2(\Gamma)$.
\end{proof}

Let $L(t)$, $t\in[a,b]$ be a continuous family of self-adjoint boundary conditions for $H$. 
By the proposition above, the family $H_{L(t)}$ is graph continuous, so its spectral flow through an arbitrary
level $\mu\in\R$ is well defined. 

\begin{prop}
\label{prop:sf=Mor} 
Suppose that the family $H_{L(t)}$ is uniformly bounded from below. Then 
\begin{equation}\label{eq:sf=Mor-Mor}
\spf_{\mu}\br{H_{L(t)}}_{t\in[a,b]} = \Mor\br{H_{L(a)}-\mu}-\Mor\br{H_{L(b)}-\mu}
\end{equation}
for every $\mu\in\R$.
\end{prop}

\begin{proof}
Let us write $H_t$ instead of $H_{L(t)}$ for brevity.
Let $c\in\R$ be less then both $\mu$ and the common lower bound of $H_t$, $t\in[a,b]$. 
Choose $t_j$ and $\eps_j$, $j=1,\ldots,N$ as in the definition of the spectral flow.
For $t\in\brr{t_{j-1},t_{j}}$ both $c$ and $\mu+\eps_{j}$ belong to the resolvent set of $H_t$,
so by \cite[thm.~IV.3.16]{Kato_perturbation} the spectral projection $\one_{[c,\,\mu+\eps_{j}]}(H_t)$ 
depends continuously on $t$. 
In particular, its rank is constant on this subinterval, so 
\[ \rank \one_{[\mu,\,\mu+\eps_{j}]}\br{H_{t_{j}}}
	 -\rank \one_{[\mu,\,\mu+\eps_{j}]}\br{H_{t_{j-1}}}
	 = \rank\one_{[c,\mu)}\br{H_{t_{j-1}}}-\rank\one_{[c,\mu)}\br{H_{t_{j}}}. \]
Summation of these equalities for $j=1,\ldots,N$ provides
\[ \spf_{\mu}\br{H_{L(t)}}_{t\in[a,b]} = \rank\one_{[c,\mu)}\br{H_a} - \rank\one_{[c,\mu)}\br{H_b} 
		= \Mor\br{H_{a}-\mu}-\Mor\br{H_{b}-\mu}, \]
which proves the proposition.
\end{proof}

\begin{rem*}
Formula \eqref{eq:sf=Mor-Mor} for smooth paths of boundary conditions appeared in \cite[thm 3.3]{LatSuk_ams20}.
Note that formula \eqref{eq:sf=Mor-Mor} does not hold without the assumption of uniform
boundedness from below, and such an assumption is not explicitly stated
in loc.~cit., though it is actually implicitly used in the proof.
As the first part of Theorem \ref{thm:sf-HBa-path} shows, this assumption cannot be omitted:
\eqref{eq:sf=Mor-Mor} fails for the family $\HBa(t)$ and the interval $[a,b]=\ii$. 
\end{rem*}


\begin{prop}\label{prop:HBa-cont}
The map $S^{1}\times\RR\ni(\alpha,t)\mapsto \HBa(t)$ is graph continuous.
\end{prop}

\begin{proof}
Let $L^{v}$ be the subspace of $\C^{4}$ given by the $\del(t)$ condition at $v\in\B$. 
It is defined by formulas (\ref{eq:delta}--\ref{eq:delta'}) 
as the common kernel of two linearly independent linear forms on $\C^{4}$. 
Each of these forms depends continuously on $(\alpha,t)$
(near $t=\infty$ one needs to pass to another coordinate $t^{-1}$, 
on which the corresponding forms still depend continuously). 
Therefore, $L^{v}$ also depends continuously on $\br{\alpha,t}$. The
vertex boundary conditions at other vertices are independent of parameters.
It remains to apply Proposition \ref{prop:HL-cont}.
\end{proof}

By this proposition, the spectral flow of $\HBa(t)$ through $\mu$ 
is well defined for $t$ running in any subinterval of $\RR$.  
We next use the spectral flow to compare the spectral counts of the original operator $H$ 
and of the decoupled operator $\HBa(\infty)$.

\begin{prop}\label{prop:sf=Mor-1} 
For every $\mu\in\R$ and $\alpha$, the following equality holds: 
\[
\Mor\br{H-\mu}-\Mor\br{\HBa(\infty)-\mu}=\begin{cases}
\spf_{\mu}\br{\HBa(t)}_{t\in\oi} & \text{for }\alpha=0\mod\pi,\\
-\spf_{\mu}\br{\HBa(t)}_{t\in\io} & \text{for }\alpha\neq0\mod\pi.
\end{cases}
\]
\end{prop}

\begin{proof}
Let $\alpha$ be fixed. By Proposition \ref{prop:bdd-below}, 
the family $\HBa(t)$ is uniformly bounded from below on the interval $t\in\oi$ if $\alpha=0$ 
and on the interval $t\in\io$ if $\alpha\neq0$. 
It remains to apply Proposition \ref{prop:sf=Mor}, taking into account that $H=\HBa(0)$
and that the values $t=\infty$ and $t=-\infty$ correspond to the same point of $\RR$.
\end{proof}

\section{Robin map and proof of Theorem \ref{thm:sf-HBa-path}}\label{sec:Robin-map}

In this section we deal with an arbitrary set $\B$ of degree two vertices of $\Gamma$. 

\begin{prop}\label{prop:f-Lambda} 
The Robin map $\LBa(\mu)\colon\C^{\B}\to\C^{\B}$ is well defined 
if and only if $\mu$ does not belong to the spectrum of $\HBa(\infty)$. 
If $\LBa(\mu)$ is well defined, then it is self-adjoint.
\end{prop}

\begin{proof}
For $\xi=0$ the boundary value problem \eqref{eq:Robin-bdry} 
describes precisely the kernel of $\HBa(\infty)-\mu$.
This proves the first statement of the proposition. 

Let $\xi,\eta$ be arbitrary vectors in $\C^{\B}$ and $f^\xi$, $f^\eta$ be the solutions 
of the corresponding boundary value problems \eqref{eq:Robin-bdry}.
Substituting $f=f^\xi$, $g=f^\eta$ in \eqref{eq:Hfg-mixed}, we obtain 
\[ \bra{\xi,\LBa(\mu)\,\eta}-\bra{\LBa(\mu)\xi,\eta} 
= \bra{Hf^\xi,f^\eta}-\bra{f^\xi,Hf^\eta} = \bra{\mu f^\xi,f^\eta}-\bra{f^\xi,\mu f^\eta} = 0, \]
which proves the self-adjointness of the Robin map.
\end{proof}

\begin{prop}\label{prop:DTN-correspondence} 
Let $\mu\notin\spec{\HBa(\infty)}$ and $t\in\R$. Then 
\begin{equation}\label{eq:dimker=}
\dim\ker\br{\HBa(t)-\mu}=\dim\ker\br{\LBa(\mu)+t}.
\end{equation}
In particular, $\dim\ker\br{H-\mu}=\dim\ker\LBa(\mu)$,
so $\LBa(\mu)$ is invertible for $\mu\notin\spec{H}$.
\end{prop}

\begin{proof}
Suppose that $\mu$ is an eigenvalue of $\HBa(t)$ with an eigenfunction $f$.
Let $\xi\in\C^{\B}$ be the $\tra $-trace
of $f$ at $\B$. The assumption $\mu\notin\spec{\HBa(\infty)}$
implies that $\xi\neq0$ (otherwise $f$ would belong to the domain
of $\HBa(\infty)$ and thus be its eigenfunction).
For every $v\in\B$ we get
\begin{equation*}
\br{\LBa(\mu)\xi}(v)=\tra'f(\vm)-\tra'f(\vp)=-t\cdot\tra f(v)=-t\xi(v),
\end{equation*}
so $-t$ is an eigenvalue of $\LBa(\mu)$ with eigenvector $\xi$. 

Conversely, assume that $-t$ is an eigenvalue of $\LBa(\mu)$
with an eigenvector $\xi\in\C^{\B}$, that is, $\LBa(\mu)\xi=-t\xi$.
Since $\mu\notin\spec{\HBa(\infty)}$, the boundary
value problem \eqref{eq:Robin-bdry}, with $\lambda$ replaced by
$\mu$, has an unique nontrivial solution $f$. 
For every $v\in\B$ we have 
\begin{equation*}
\tra'f(\vp)-\tra'f(\vm) = \br{-\LBa(\mu)\xi}(v)=t\xi(v)=t\cdot\tra f(v),
\end{equation*}
so $f$ belongs to the domain of $\HBa(t)$
and is its eigenfunction with eigenvalue $\mu$. 

The reasoning above shows that the eigenfunctions $f$ and the eigenvectors
$\xi$ are in one-to-one correspondence, which implies equality \eqref{eq:dimker=}.
Substituting $t=0$ in \eqref{eq:dimker=}, we obtain the last statement of the proposition. 
\end{proof}

\subsection*{Proof of Theorem \ref{thm:sf-HBa-path}. }

By Proposition \ref{prop:monotonicity}, every spectral curve of $\HBa(t)$ 
is either constant or strictly increasing. 
Therefore, for every interval $[a,b]\subset\RR$ the spectral flow of $\HBa(t)$, $t\in[a,b]$ through $\mu$ 
is equal to the total number of intersections of the level $\mu$
by increasing spectral curves as $t$ is running the interval $(a,b]$.
(An increasing spectral curve $\lambda(t)$ with $\lambda(a)=\mu$,
resp. $\lambda(b)=\mu$ contributes 0, resp. 1 to the spectral flow).

The spectral flow of a loop does not depend on the choice of a spectral level $\mu$; 
let us choose $\mu\in\R$ outside of the spectrum of $\HBa(\infty)$. 
Then the spectral flow through $\mu$ of $\HBa(t)$ with $t$ running $\RR$ (or, what is the same, $\ii$)
is equal to the total number of intersections of the level $\mu$ by increasing spectral curves for $t\in\R$. 
By Proposition \ref{prop:DTN-correspondence}, this number is equal to the total number 
of real eigenvalues of $\LBa(\mu)$, counted with multiplicities. 
The operator $\LBa(\mu)$ is self-adjoint, so this number is equal to the dimension 
of the vector space $\C^{\B}$ on which $\LBa(\mu)$ acts. 
This proves the first statement of the theorem.

Let now $\mu\in\R$ lie outside of the spectrum of both $H$ and $\HBa(\infty)$.
Then the spectral flow through $\mu$ for $t$ running $\io$ is equal 
to the number of intersections of the level $\mu$ by increasing spectral curves for $t\in(-\infty,0)$,
which in turn is equal to $\Pos\br{\LBa(\mu)}$ by Proposition \ref{prop:DTN-correspondence}.
Similarly, the spectral flow along $\oi$ is equal to the number of intersections of level $\mu$ 
for $t$ running $(0,\infty)$, which is equal to $\Mor\br{\LBa(\mu)}$.
This completes the proof of the theorem.
\qed\vspace{3pt}

\begin{rem*}
If $\mu\in\spec{H}$, then the right hand side of \eqref{eq:sf=Pos} should be increased by $\dim\ker\br{H-\mu}$.
\end{rem*}

\section{Proof of Theorems \ref{thm:nodal-def}, \ref{thm:Robin-def}, \ref{thm:Mor-Robin}, and \ref{thm:sf=sf}}
\label{sec:Proof-of-index}

\subsection*{Betti number.}

The first Betti number (the number of independent cycles) of a connected
graph $\Gamma$ is given by the formula 
\[ \beta(\Gamma)=\abs{\E}-\abs{\V}+1. \]
The last summand $1$ in this formula stands for the \q{one connected component} of $\Gamma$; 
for a non-connected graph, this summand should be replaced by the number of connected components. 

Let $\B$ be a subset of $\V$ including only vertices of degree two
and $\Gamma_{\B}=\overline{\Gamma\backslash\B}$ be the graph obtained by cutting $\Gamma$ at points of $\B$. 
Cutting $\Gamma$ at one vertex of degree two increases the number of vertices
by one and leaves the number of edges unchanged, 
so $\abs{\V(\Gamma_{\B})}=\abs{\V}+\abs{\B}$ and $\abs{\E(\Gamma_{\B})}=\abs{\E}$. Therefore,
\begin{equation}\label{eq:beta-beta}
\beta(\Gamma)-\beta(\Gamma_{\B}) = \abs{\B}-\abs{\pi_{0}(\Gamma_{\B})}+1,
\end{equation}
where $\abs{\pi_{0}(\Gamma_{\B})}$ is the number of connected components of $\Gamma_{\B}$. 

If $\B=\Pa(f)$ is the set of $\alpha$-Robin points of $f$, 
then the connected components of the cutted graph
\[ \Gaf\coloneqq\Gamma_{\Pa(f)}=\overline{\Gamma\backslash\Pa(f)} \]
are $\alpha$-Robin domains of $f$, so 
\begin{equation}\label{eq:beta-beta_alpha}
\beta(\Gamma)-\beta(\Gaf) = \Paf-\nua(f)+1.
\end{equation}

\subsection*{Deformation at Robin points of an eigenfunction.}

Let $(\lambda,f)$ be an eigenpair of the Neumann-Kirchhoff Laplacian $H$.
The spectral counts $n(\lambda)$ and $N(\lambda)$ may be expressed 
in terms of the Morse indices of the shifted operator $H$:
\begin{equation}\label{eq:MorH}
	\Mor\br{H-(\lambda+\eps)} = N(\lambda) \txt{and} \Mor\br{H-(\lambda-\eps)} = n(\lambda)-1
\end{equation}
for $\eps>0$ small enough.
In this section we compare these two indices with the Morse indices of $\Ha^f(\infty)-(\lambda\pm\eps)$, 
where $\Ha^{f}(t)$ is the Laplacian 
with $\del(t)$ conditions at $\alpha$-Robin points of $f$ 
and Neumann-Kirchhoff conditions at all other vertices.

By the definition, each Robin point $v\in\Pa(f)$ is either an internal point of an edge 
or a vertex of degree two and $\tra f(\vp) = \tra f(\vm)=0$. 
Therefore, $f$ belongs to the domain of $\HBa(\infty)$ and is its eigenfunction. 
But $f$ is also an eigenfunction of $H=\Ha^f(0)$,
so Proposition \ref{prop:constant-curve} provides the following result.

\begin{prop}\label{lem:f-t} 
If $(\lambda,f)$ is an eigenpair of $H$, then it is also an eigenpair of $\Ha^f(\infty)$
and, moreover, of $\Ha^f(t)$ for every $t\in\RR$.
In other words, the family $\Ha^f(t)$ has a constant spectral curve on the level $\lambda$.
\end{prop}

\subsection*{Decoupled operator. }

The vertex boundary condition of the operator $\HBa(\infty)$
at $v\in\B$ is given by the Robin condition at each side of $v$,
so the two sides of $v$ are in fact disconnected. Therefore, $\HBa(\infty)$
can be considered as the Laplacian on the graph $\Gamma_{\B}=\overline{\Gamma\backslash\B}$
obtained by cutting the original graph $\Gamma$ at points of $\B$.
In other words, $\HBa(\infty)$ is the orthogonal
sum of operators $H^{i}$ defined on the connected components $\Gamma_{\B}^{i}$
of $\Gamma_{\B}$:
\[ \Gamma_{\B}=\sqcup_{i}\Gamma_{\B}^{i} \quad \HBa(\infty)=\oplus_{i}H^{i}. \]
Each \q{internal} vertex of $\Gamma_{\B}^{i}$ is equipped with the Neumann-Kirchhoff boundary condition, 
while the \q{external} vertices of $\Gamma_{\B}^{i}$ (those that arise as a result of cutting)
have degree one and are equipped with the $\alpha$-Robin condition $\tra f(v)=0$. 

When $\B=\Pa(f)$ is the set of $\alpha$-Robin points of $f$,
the operator $\Ha^{f}(\infty)$ is defined on the disjoint union 
$\Ga = \sqcup_{i}\Ga^{i}$ of $\alpha$-Robin domains $\Ga^{i}$, $i=1,\ldots,\nua(f)$.
Denote by $\Ha^{i}$ the restriction of $\Ha^{f}(\infty)$ to $\Ga^{i}$; then
\begin{align}
\ker\br{\Ha^f(\infty)-\lambda} &= \bigoplus_{i=1}^{\nua}\ker\br{\Ha^{i}-\lambda},
\label{eq:dim-sum} \\
\Mor\br{\Ha^f(\infty)-\lambda} &= \sum_{i=1}^{\nua}\Mor\br{\Ha^{i}-\lambda}.
\label{eq:dim<sum}
\end{align}

\subsection*{Proof of Theorem \ref{thm:nodal-def} and Theorems \ref{thm:Mor-Robin}, \ref{thm:sf=sf} for $\alpha=0$.}

Let $(\lambda,f)$ be an eigenpair of the Neumann-Kirchhoff Laplacian $H$.
Then it is also an eigenpair of $H_0^f(\infty)$.
The restriction $f_{i}$ of $f$ to the nodal domain $\Gamma_{0}^{i}$ is an eigenfunction of $H_{0}^{i}$ 
which does not vanish at points other than vertices of degree one. 
It follows that $\lambda$ is a simple ground state of $H_{0}^{i}$, that is, 
$\dim\ker\br{H_{0}^{i}-\lambda}=1$ and $\Mor\br{H_{0}^{i}-\lambda}=0$. 
Substituting this in (\ref{eq:dim-sum}--\ref{eq:dim<sum}), we see that 
$\lambda$ is a ground state of $H_{0}^{f}(\infty)$ of multiplicity $\nu_{0}(f)$, that is, 
\begin{equation}\label{eq:ker-Mor-0}
\dim\ker\br{H_{0}^{f}(\infty)-\lambda} = \nu_{0}(f) \txt{and} \Mor\br{H_{0}^{f}(\infty)-\lambda}=0.
\end{equation}
Choose $\eps>0$ such that the intervals $[\lambda-\eps,\lambda)$ and $(\lambda,\lambda+\eps]$
belong to the resolvent set of both $H=H_{0}^{f}\br{0}$ and $H_{0}^{f}(\infty)$. 
Then \eqref{eq:ker-Mor-0} can be written as follows:
\begin{equation}\label{eq:Mor-0-pm}
\Mor\br{H_{0}^{f}(\infty)-(\lambda+\eps)} = \nu_{0}(f) \txt{and} \Mor\br{H_{0}^{f}(\infty)-(\lambda-\eps)}=0.
\end{equation}
Substituting \eqref{eq:MorH} and \eqref{eq:Mor-0-pm} in Proposition \ref{prop:sf=Mor-1},
we obtain 
\begin{equation}\label{eq:sf0lambda+-eps}
  \spf_{\lambda+\eps}\br{H_0^f(t)}_{t\in\oi} =  N(\lambda)-\nu_{0}(f) \txt{and}
  \spf_{\lambda-\eps}\br{H_0^f(t)}_{t\in\oi} = n(\lambda)-1. 
\end{equation}
By Propositions \ref{prop:f-Lambda} and \ref{prop:DTN-correspondence}, 
the DtN operators $\Lambda_{0}^{f}(\lambda\pm\eps)$ are well defined and invertible.
By the second part of Theorem \ref{thm:sf-HBa-path}, the Morse indices of these operators
are equal to the spectral flows along $\oi$, so 
\begin{equation}\label{eq:Mor0lambda+-eps}
 \Mor\br{\Lambda_{0}^{f}(\lambda+\eps)} = N(\lambda)-\nu_{0}(f) \txt{and}
 \Mor\br{\Lambda_{0}^{f}(\lambda-\eps)} = n(\lambda)-1,
\end{equation}
which proves the first part of Theorem \ref{thm:nodal-def}.
Since the sum of negative and positive indices of $\Lambda_{0}^{f}(\lambda-\eps)$ is equal to $\Pof$,
the second equality implies
\begin{equation}\label{eq:sf0=Pos:lambda-eps}
	\spf_{\lambda-\eps}\br{H_{0}^{f}(t)}_{t\in\io} = \Pos\br{\Lambda_{0}^{f}(\lambda-\eps)} = \Pof-n(\lambda)+1.
\end{equation}
Taking into account \eqref{eq:beta-beta_alpha}, we obtain
\[ n(\lambda)-\nu_{0}(f) = \Pof - \nu_{0}(f) + 1 - \Pos\br{\Lambda_0^{f}(\lambda-\eps)}
= \beta(\Gamma) - \beta(\Gamma_{0}) - \Pos\br{\Lambda_0^{f}(\lambda-\eps)}, \]
which completes the proof of Theorem \ref{thm:nodal-def}.

For $\alpha=0$ Theorem \ref{thm:sf=sf} is given by \eqref{eq:sf0=Pos:lambda-eps} 
and the first equality of \eqref{eq:sf0lambda+-eps},
while Theorem \ref{thm:Mor-Robin} is given by \eqref{eq:sf0=Pos:lambda-eps} 
and the first equality of \eqref{eq:Mor0lambda+-eps}.
\qed

\subsection*{Large values of $\lambda$ and proof of Corollary \ref{cor:nodal-def-1}.}

In the rest of the section we suppose that an eigenpair $(\lambda,f)$ satisfies Assumption \ref{assu:generic}. 
In particular, $\lambda>\br{\pi/\lmin}^{2}$. 
This condition on $\lambda$ implies that, for every $\alpha$, 
there is at least one $\alpha$-Robin point inside each edge of $\Gamma$. 
Therefore, each $\alpha$-Robin domain $\Ga^{i}$ is either an interval or a star graph,
$\beta(\Ga)=0$, and 
\begin{equation}\label{eq:beta=P-nu+1}
\beta(\Gamma) = \Paf-\nua(f)+1.
\end{equation}
In particular, $\beta(\Gamma_0)=0$ and $\beta(\Gamma) = \Pof-\nu_0(f)+1$. 
Substituting this in \eqref{eq:nodal-def}, we obtain Corollary \ref{cor:nodal-def-1}.

\subsection*{Proof of Theorems \ref{thm:Robin-def}, \ref{thm:Mor-Robin}, and \ref{thm:sf=sf} for $\alpha\neq 0$. }

We already proved these theorems for $\alpha=0$.
Suppose now that $\alpha\neq0\mod\pi$.
Let $(\lambda,f)$ be an eigenpair of $H$ satisfying  Assumption \ref{assu:generic}.
Then each $\alpha$-Robin domain $\Ga^{i}$ is either an interval or a star graph, as above,
and $f$ does not vanish at a vertex of degree larger than two. 
By \cite[cor. 3.1.9]{BerKuc_graphs}, the eigenvalue $\lambda$ of $\Ha^{i}$ is simple, 
so \cite[thm. 2.6]{Ber_cmp08} implies 
\[ \Mor\br{\Ha^{i}-\lambda}=\nu_{0}\br{f_{i}}-1=\abs{\P_{0}\br{f_{i}}}, \]
where $f_{i}$ is the restriction of $f$ to $\Ga^{i}$.
Since $\alpha\neq0$, the nodal set $\P_{0}(f)$ is disjoint from $\Pa(f)$ 
(otherwise $f$ and $f'$ would vanish simultaneously at some internal point
and thus on the whole edge, which contradicts assumption of the theorem),
so $\P_{0}(f)=\cup\P_{0}\br{f_{i}}$.
Therefore, (\ref{eq:dim-sum}--\ref{eq:dim<sum}) can be written in this case as
\begin{align}
 & \dim\ker\br{\Ha^{f}(\infty)-\lambda} = \sum_{i=1}^{\nua}\dim\ker\br{\Ha^{i}-\lambda} = \nua(f), 
 \label{eq:dimKer-alpha} \\
 & \Mor\br{\Ha^{f}(\infty)-\lambda} = \sum_{i=1}^{\nua}\Mor\br{\Ha^{i}-\lambda}
   = \sum_{i=1}^{\nua}\abs{\P_{0}\br{f_{i}}}=\Pof. \label{eq:dimMor-alpha}
\end{align}
The rest of the proof is similar to the proof of the case $\alpha=0$, 
but the roles of the intervals $\oi$ and $\io$ are interchanged.
Choose $\eps>0$ such that the intervals $[\lambda-\eps,\lambda)$ and $(\lambda,\lambda+\eps]$
belong to the resolvent set of both $H=\Ha^{f}\br{0}$ and $\Ha^{f}(\infty)$. 
Then (\ref{eq:dimKer-alpha}--\ref{eq:dimMor-alpha}) can be written as follows:
\begin{equation*}
\Mor\br{\Ha^f(\infty)-(\lambda+\eps)} = \Pof+\nua(f) \txt{and} \Mor\br{\Ha^f(\infty)-(\lambda-\eps)} = \Pof.
\end{equation*}
Substituting this and \eqref{eq:MorH} in Proposition \ref{prop:sf=Mor-1}, we obtain 
\begin{align}
& \spf_{\lambda+\eps}\Ha^f(t)_{t\in\io} = \Pof+\nua(f)-N(\lambda), \label{eq:sfa-lambda+eps} \\
& \spf_{\lambda-\eps}\Ha^f(t)_{t\in\io} = \Pof-n(\lambda)+1. \label{eq:sfa-lambda-eps} 
\end{align}
By Propositions \ref{prop:f-Lambda} and \ref{prop:DTN-correspondence}, 
the Robin operators $\La^f(\lambda\pm\eps)$ are well defined and invertible.
By the second part of Theorem \ref{thm:sf-HBa-path}, 
the positive indices of these operators are equal to the spectral flows along $\io$, 
so (\ref{eq:sfa-lambda+eps}--\ref{eq:sfa-lambda-eps}) imply
\begin{align}
& \Pos\br{\La^f(\lambda+\eps)} = \Pof+\nua(f)-N(\lambda), \label{eq:Pos-lambda+eps} \\
& \Pos\br{\La^f(\lambda-\eps)} = \Pof-n(\lambda)+1. \label{eq:Pos-lambda-eps}
\end{align}
Equality \eqref{eq:Pos-lambda+eps} is just another form of \eqref{eq:Robin-count-Pos}, 
so it proves Theorem \ref{thm:Robin-def} for $\alpha\neq 0$.

The sum of negative and positive indices of $\La^f(\lambda+\eps)$ is equal to $\Paf$,
so \eqref{eq:Pos-lambda+eps} can be written as
\begin{equation*}
	\Mor\br{\La^f(\lambda+\eps)} = \Paf-\Pof-\nua(f)+N(\lambda).
\end{equation*}
By \eqref{eq:beta=P-nu+1} $\Paf-\nua(f) = \Pof-\nu_0(f)$, hence
\begin{equation}\label{eq:sf=Mor:lambda+eps}
	\spf_{\lambda+\eps}\br{\Ha^f(t)}_{t\in\oi} = \Mor\br{\La^f(\lambda+\eps)} 
	  = N(\lambda)-\nu_0(f).
\end{equation}
Equalities \eqref{eq:sfa-lambda-eps} and \eqref{eq:sf=Mor:lambda+eps}
prove Theorem \ref{thm:sf=sf} for $\alpha\neq 0$,
while \eqref{eq:Pos-lambda-eps} and \eqref{eq:sf=Mor:lambda+eps}
prove Theorem \ref{thm:Mor-Robin} for $\alpha\neq 0$.
\qed

\section{Differences of Dirichlet, Robin, and Neumann counts}\label{sec:sf-DRN}

\subsection*{Paths.}

Let again $(\lambda,f)$ be an eigenpair of the Neumann-Kirchhoff Laplacian $H$.
There are two natural ways to connect the operator $H^f_0(\infty)=H^f_0(-\infty)$ 
with the Neumann-Kirchhoff Laplacian $H = H^f_0(0)$.
Both are given by the family $H^f_0(t)$, but in the first case $t$ is running the interval $\io$, 
while in the second case $t$ is running the interval $\oi$ in the \emph{opposite direction}, 
from $\infty$ to $0$.
We write these paths symbolically as $H_0^f(-\infty) \darr H$ and $H_0^f(\infty) \darr H$.

Similarly, there are two natural ways to connect $H = \Ha^f(0)$ with $\Ha^f(\infty)=\Ha^f(-\infty)$: 
both are given by the family $\Ha^f(t)$,
with $t$ running either the interval $\oi$ in the positive direction, 
or the interval $\io$ in the negative direction (from $0$ to $-\infty$).
We write these paths as $H \darr \Ha^f(\infty)$ and $H \darr \Ha^f(-\infty)$.

Concatenating a path going from $H^f_0(\infty)$ to $H$ with a path going from $H$ to $\Ha^f(\infty)$,
we obtain a path going from $H^f_0(\infty)$ to $\Ha^f(\infty)$ via $H$.
Every half of this combined path can be chosen from two alternatives, 
so there are four natural choices for the combined path.
We write them as $H^f_0(\pm\infty) \darr H \darr \Ha^f(\pm\infty)$ (with the corresponding choice of signs),
or shorter as $H^f_0(\pm\infty) \darr \Ha^f(\pm\infty)$, 
and depict them by arrows as shown on pictures below.

\subsection*{Spectral flow.}

Suppose now that $(\lambda,f)$ \emph{satisfies Assumption \ref{assu:generic}} and $\eps>0$ is small enough.
Taking into account the additivity of the spectral flow under concatenation of paths, 
we see that Theorem \ref{thm:sf=sf} can be written as follows.
\begin{align}
& \spf_{\lambda+\eps}\br{H^f_0(\infty) \darr \Ha^f(\infty)} = 0, \label{eq:sf:lambda+eps=0} \\
& \spf_{\lambda-\eps}\br{H^f_0(-\infty) \darr \Ha^f(-\infty)} = 0. \label{eq:sf:lambda-eps=0}
\end{align}
Using these equalities, we can now easily compute the spectral flow 
through $\lambda\pm\eps$ of the other three paths.
Indeed, in the combined path
\begin{equation}\label{eq:0-iHiHi}
	H^f_0(-\infty) \darr H \darr \Ha^f(\infty) \darr H \darr H^f_0(\infty) 
\end{equation}
the middle part $H \darr \Ha^f(\infty) \darr H$ is contractible, 
so \eqref{eq:0-iHiHi} is homotopic to the path $H^f_0(t)$, $t\in\ii$, 
and the first part of Theorem \ref{thm:sf-HBa-path} provides
\[ \spf_{\mu}\br{H^f_0(-\infty) \darr \Ha^f(\infty)} - \spf_{\mu}\br{H^f_0(\infty) \darr \Ha^f(\infty)} 
  =  \spf_{\mu}\br{H_0^f(t)}_{t\in\ii} = \Pof \]
for every $\mu\in\R$.	
Taking this together with \eqref{eq:sf:lambda+eps=0}, we obtain 
\begin{equation}\label{eq:sf:lambda+eps-+}
	\spf_{\lambda+\eps}\br{H^f_0(-\infty) \darr \Ha^f(\infty)} = \Pof. 
\end{equation}
Similarly, 
\begin{align}
& \spf_{\lambda+\eps}\br{H^f_0(\infty) \darr \Ha^f(-\infty)} = -\Paf, \label{eq:sf:lambda+eps+-} \\
& \spf_{\lambda+\eps}\br{H^f_0(-\infty) \darr \Ha^f(-\infty)} = \Pof-\Paf = \nu_0(f)-\nu_{\alpha}(f), 
\label{eq:sf:lambda+eps--}
\end{align}
where the last equality follows from \eqref{eq:beta=P-nu+1}.
In a similar manner, \eqref{eq:sf:lambda-eps=0} 
together with the first part of Theorem \ref{thm:sf-HBa-path} provide
\begin{align}
& \spf_{\lambda-\eps}\br{H^f_0(-\infty) \darr \Ha^f(\infty)} = \Paf, 
\label{eq:sf:lambda-eps-+} \\
& \spf_{\lambda-\eps}\br{H^f_0(\infty) \darr \Ha^f(-\infty)} = -\Pof, 
\label{eq:sf:lambda-eps+-} \\
& \spf_{\lambda-\eps}\br{H^f_0(\infty) \darr \Ha^f(\infty)} = \Paf-\Pof = \nu_{\alpha}(f)-\nu_0(f).
\label{eq:sf:lambda-eps++}
\end{align}
Formulas \eqref{eq:sf:lambda+eps--} and \eqref{eq:sf:lambda-eps++}
express the difference of the Robin and nodal counts of $f$ in terms 
of the spectral flow of a path connecting the Laplacian on $\Gamma_0$ 
with the Laplacian on $\Ga$, both taken with appropriate boundary conditions.

For $\alpha=\nicefrac{\pi}{2}$ we obtain the difference of the Neumann and nodal counts of $f$
expressed in terms of the spectral flow of a path connecting the Hamiltonian on $\Gamma_0$ 
with the Neumann-Kirchhoff Laplacian $H(\Ga)$ on $\Ga$.
The importance of this difference was pointed out in \cite{MR4314130}.

\subsection*{Symmetry.}

The following two pictures show schematically all these eight paths and their spectral flows (in blue color):
\begin{equation*}
\begin{tikzcd}[cramped]
{\boxed{\boldsymbol{\spf_{\lambda+\eps}}}}
& \Ha^f(\infty) \ar[dash]{d} & 
\\   
 H_0^f(-\infty) \ar[dash]{r} 
  \ar[blue, dashed, bend right=30, start anchor={[yshift=1ex]}, end anchor={[xshift=0]}]{ru}{\abs{\P_0}} 
	\ar[blue, dashed, bend left=30, start anchor={[yshift=-1ex]}, end anchor={[xshift=0]}]{rd}[swap]{\abs{\P_0}-\abs{\Pa}}
  & H \ar[dash]{r} 
	& H_0^f(\infty) 
	\ar[blue, dashed, bend left=30, start anchor={[yshift=1ex]}, end anchor={[xshift=-0.2ex]}]{lu}[swap]{0} 
	\ar[blue, dashed, bend right=30, start anchor={[yshift=-1ex]}, end anchor={[xshift=-0.2ex]}]{ld}{-\abs{\Pa}}
\\
& \Ha^f(-\infty) \ar[dash]{u} &
\end{tikzcd}
\quad\quad\quad
\begin{tikzcd}[cramped]
{\boxed{\boldsymbol{\spf_{\lambda-\eps}}}}
& \Ha^f(\infty) \ar[dash]{d} & 
\\  
 H_0^f(-\infty) \ar[dash]{r} 
  \ar[blue, dashed, bend right=30, start anchor={[yshift=1ex]}, end anchor={[xshift=0]}]{ru}{\abs{\Pa}} 
	\ar[blue, dashed, bend left=30, start anchor={[yshift=-1ex]}, end anchor={[xshift=0]}]{rd}[swap]{0}
  & H \ar[dash]{r} 
	& H_0^f(\infty) 
	\ar[blue, dashed, bend left=30, start anchor={[yshift=1ex]}, end anchor={[xshift=-0.2ex]}]{lu}[swap]{\abs{\Pa}-\abs{\P_0}} 
	\ar[blue, dashed, bend right=30, start anchor={[yshift=-1ex]}, end anchor={[xshift=-0.2ex]}]{ld}{-\abs{\P_0}}
\\
& \Ha^f(-\infty) \ar[dash]{u} &
\end{tikzcd}
\end{equation*}
One can see a surprising antisymmetry here:
the blue parts of two pictures go one into another after the central reflection and change of the signs, that is,
\begin{equation}\label{eq:antisymmetry}
	\spf_{\lambda+\eps}\br{H^f_0(\I) \darr \Ha^f(\I')} = -\spf_{\lambda-\eps}\br{H^f_0(-\I) \darr \Ha^f(-\I')}, 
\end{equation}
where $\I$ and $\I'$ can take the values $\infty$ or $-\infty$ independently one from another.
The authors do not know if there is a deeper meaning behind this observation.

\section{Spectral flow and Betti number}\label{sec:sf-beta}

In this section we present two results, which relate the Betti number of the graph 
with the spectral flow of appropriate operator families.

\subsection*{Deformation to Neumann conditions.}

Let $\alpha=\nicefrac{\pi}{2}$ and $\B\subset\V$ be a set of degree two vertices.
The boundary condition $\delta_{\pi/2}^v(\infty)$ means the vanishing of $f'$ at each side of $v$.
Therefore, $H_{\pi/2}^{\B}(\infty) = H(\Gamma_\B)$ is the Neumann-Kirchhoff Laplacian  
on the graph $\Gamma_\B=\overline{\Gamma\backslash\B}$ obtained by cutting $\Gamma$ at points of $\B$.

The one-parameter family $H_{\pi/2}^{\B}(t)$, $t\in\oi$
connects the Neumann-Kirchhoff Lap\-lacian $H=H(\Gamma)$ with the Neumann-Kirchhoff Laplacian $H(\Gamma_\B)$. 
The following result relates the spectral flow of this family 
(or equivalently the Morse index of the two-sided Neumann-to-Dirichlet operator)
with the number of cycles removed from $\Gamma$ in the process of cutting.
See also Fig.~\ref{fig:Demo-betti}.

\begin{thm}\label{thm:beta-beta} 
Let $H$ be the Neumann-Kirchhoff Laplacian on $\Gamma$. Then 
\begin{equation*}
 \Mor\br{\Lambda_{\pi/2}^{\B}(\eps)} =  \spf_{\eps}\br{H_{\pi/2}^{\B}(t)}_{t\in\oi}
  = \beta(\Gamma)-\beta(\Gamma_\B) \;\text{ for }\eps>0\text{ small enough}.
\end{equation*}
\end{thm}

\begin{proof}
The first equality is given by the second part of Theorem \ref{thm:sf-HBa-path},
so we only need to prove the second equality.
By Proposition \ref{prop:sf=Mor-1},  
\[ \spf_{\eps}\br{H_{\pi/2}^{\B}(t)}_{t\in\io}=\Mor\br{H_{\pi/2}^{\B}(\infty)-\eps}-\Mor\br{H-\eps}. \]
The Neumann-Kirchhoff Laplacian $H$ on the connected graph $\Gamma$ is non-negative,
with one-dimensional kernel consisting of constant functions, so 
\[ \Mor\br{H-\eps} = \dim\ker(H) = 1. \]
The Neumann-Kirchhoff Laplacian $H_{\pi/2}^{\B}(\infty) = H(\Gamma_\B)$ 
on the graph $\Gamma_{\B}$ is also non-negative; 
its kernel consists of functions which are constant on connected components of $\Gamma_{\B}$.
Hence 
\[ \Mor\br{H_{\pi/2}^{\B}(\infty)-\eps} = \dim\ker\br{H_{\pi/2}^{\B}(\infty)}
  = \abs{\pi_{0}\br{\Gamma_{\B}}}, \]
the number of connected components of $\Gamma_{\B}$.
Taking all this together, we obtain
\[ \spf_{\eps}\br{H_{\pi/2}^{\B}(t)}_{t\in\io} = \abs{\pi_{0}\br{\Gamma_{\B}}}-1. \]
Combining this formula with the first part of Theorem~\ref{thm:sf-HBa-path}
and using the additivity of the spectral flow under concatenation, we obtain
\[
\spf_{\eps}\br{H_{\pi/2}(t)}_{t\in\oi} = \spf_{\eps}(\cdot)_{t\in\ii}-\spf_{\eps}(\cdot)_{t\in\io}
 = \abs{\B}-\abs{\pi_{0}(\Gamma_{\B})}+1.
\]
It remains to apply equality \eqref{eq:beta-beta}. 
\end{proof}

\begin{figure}
\includegraphics[scale=0.55]{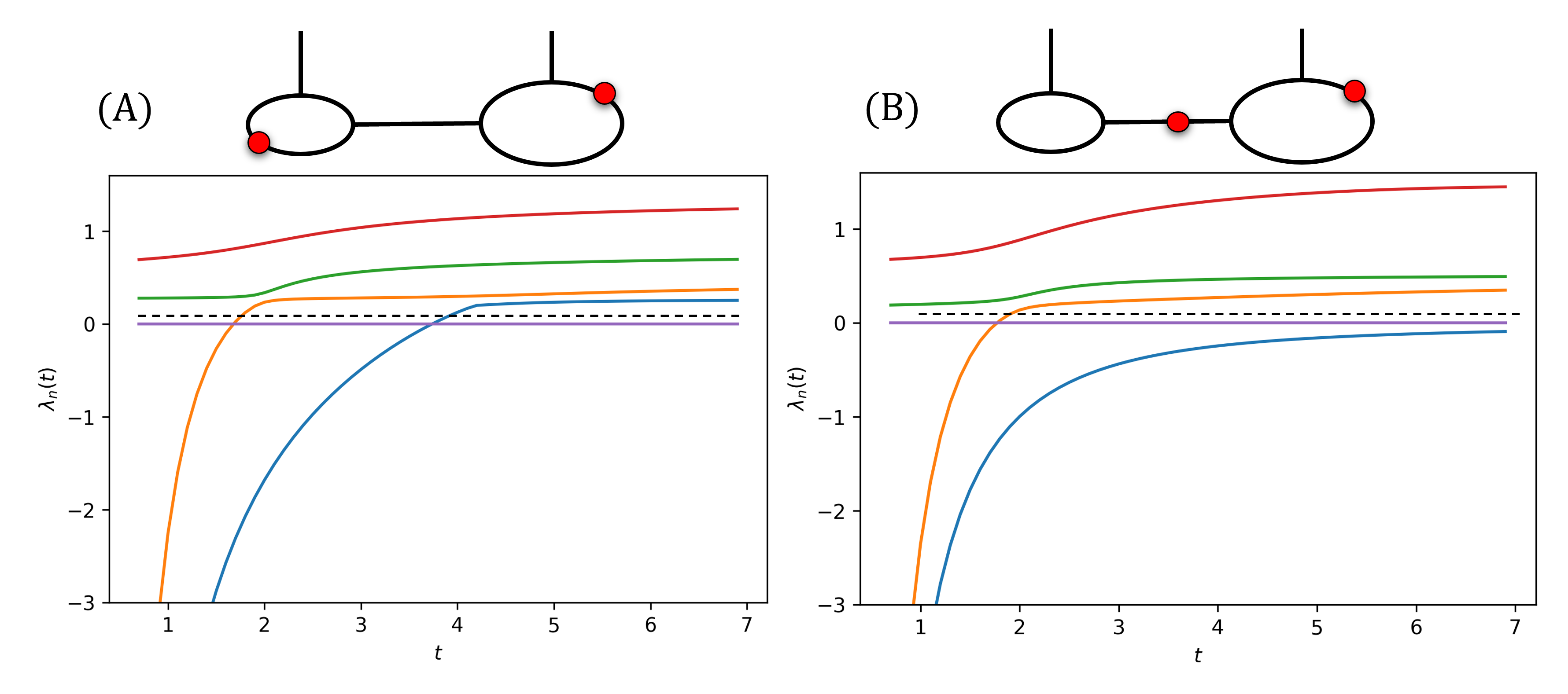}
\caption[Theorem \ref{thm:beta-beta}]
{\small{Demonstration of Theorem \ref{thm:beta-beta}
for a two-cycle graph and two different choices of $\B$. 
In (A), the set $\B$ is chosen so that $\beta(\Gamma_{\B})=0$,
and thus the spectral flow through the dashed line $\lambda=\eps$ is $\beta(\Gamma)-\beta(\Gamma_{\B})=2$.
In (B), the set $\B$ is chosen so that $\beta(\Gamma_{\B})=1$,
and thus the spectral flow through the dashed line $\lambda=\eps$ is $\beta(\Gamma)-\beta(\Gamma_{\B})=1$.}
\label{fig:Demo-betti}}
\end{figure}

\subsection*{Counting intersections along $\R$.}

Our next result expresses the Betti number of $\Gamma$ in terms of the spectral flow of a specific operator family.

\begin{thm}\label{thm:SF-R-beta} 
Let $(\lambda,f)$ be an eigenpair of the Neumann-Kirchhoff Laplacian $H$.  
Suppose that $(\lambda,f)$ satisfies Assumption \ref{assu:generic} and the eigenvalue $\lambda$ is simple. 
Then, for a fixed $\alpha$ and $T>0$ large enough, 
\begin{equation*}
\spf_{\lambda}\br{\Ha^{f}(t)}_{t\in\brr{-T,T}}=\beta(\Gamma). 
\end{equation*}
\end{thm}

\begin{proof}
By Proposition \ref{lem:f-t}, the family $\Ha^f(t)$ has a constant spectral curve on the level $\lambda$. 
Since $\lambda$ is a simple eigenvalue of $H=\Ha^{f}\br{0}$, it follows
from Proposition~\ref{prop:monotonicity} that all other spectral
curves intersecting the spectral level $\lambda$ are strictly increasing.

By the first part of Theorem~\ref{thm:sf-HBa-path}, the number of
these strictly increasing spectral curves is $\Paf$.
Formulas \eqref{eq:ker-Mor-0} for $\alpha=0$ and \eqref{eq:dimKer-alpha} for $\alpha\neq0$ provide 
\[ \dim\ker\br{\Ha^{f}(\infty)-\lambda}=\nua(f). \]
Hence, in addition to the constant spectral curve at the level $\lambda$, 
there are exactly $\nua(f)-1$ strictly increasing spectral curves passing through $\lambda$ at $t=\infty$. 
The rest 
\begin{equation*}
	\Paf-\br{\nua(f)-1}=\beta(\Gamma) 
\end{equation*}
of strictly increasing spectral curves intersect the level $\lambda$ at values of parameter $t\neq\infty$, 
that is, $t\in\R$.  
Since the number of intersections is finite, 
all of them happen in the interval $t\in\brr{-T,T}$ for $T$ large enough.
\end{proof}

\begin{figure}
\includegraphics[scale=0.5]{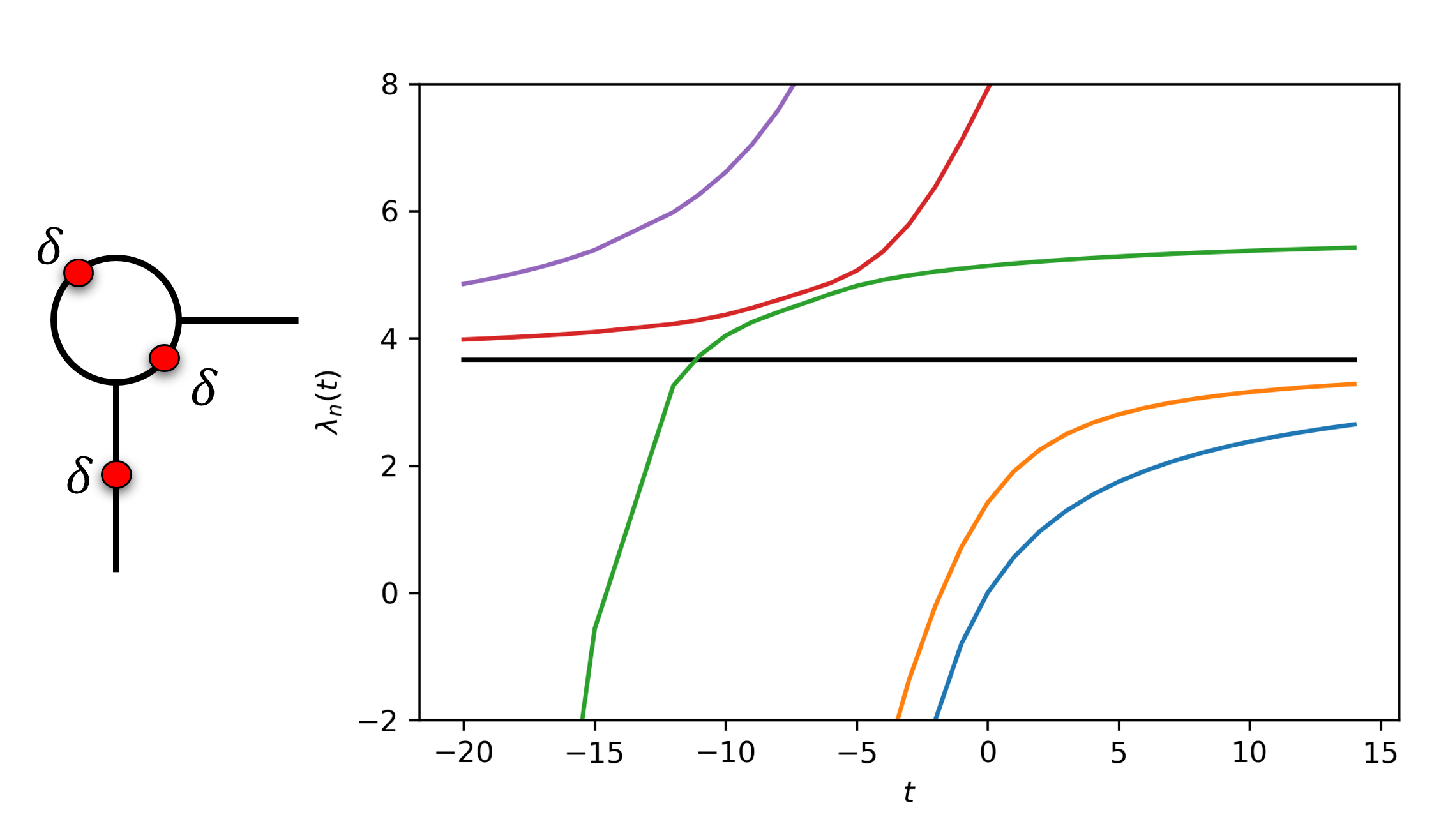}
\caption[Theorem \ref{thm:SF-R-beta}]
{\small{Demonstration of Theorem \ref{thm:SF-R-beta} for a graph with one cycle and $\alpha=0$.
An eigenfunction $f$ has three $\alpha$-Robin points (colored red), that is, $\Paf=3$.
There is one horizontal spectral curve of $\Ha^f(t)$ at the level $\lambda$ (marked in black),
one increasing spectral curve crossing this level at a finite value of $t$,
and two increasing spectral curves crossing this level at infinity.
The total spectral flow along $\ii$ is equal to $3$, as stated in Theorem \ref{thm:sf-HBa-path}.}
\label{fig:Demo-betti-2-1}} 
\end{figure}

\appendix

\section*{Appendix. Spectral flow and winding number}

\renewcommand{\thesection}{A} 

This appendix is devoted to the proof of the \q{spectral flow = winding number} formula 
for loops of boundary conditions on a quantum graph. 
See Theorem \ref{thm:sf-wind-loop}.
We then apply this result to provide an alternative proof of the first part of Theorem \ref{thm:sf-HBa-path}.
Generalization of these results, as well as a more detailed exposition, 
can be found in the paper of the second author \cite{prokhorova_quantum}.

\subsection*{Unitary form of boundary conditions.}

For a Hilbert space $V$, we will denote by $\U(V)$ the set of unitary operators $V\to V$
and equip it with the norm topology.

A self-adjoint boundary condition for $H$ at a vertex $v$ can be written as follows:
\begin{equation}\label{eq:Uv}
	i(U^v-1)\fvec(v) + (U^v+1)\nabla\fvec(v) = 0, \txt{where} U^v\in\U(\C^{\E_v}),
\end{equation}
$\fvec(v)\in\C^{\E_v}$ is the vector of boundary values $f_e(v)$, 
and $\nabla\fvec(v)\in\C^{\E_v}$ is the vector of outgoing derivatives $\nabla f_e(v)$, $e\in\E_v$.
See \cite[thm.~1.4.4]{BerKuc_graphs}.
A (local) boundary condition for $H$ is given by a tuple $(U^v)_{v\in\V}$,
or equivalently by the unitary operator $U=\oplus_v U^v\in\U(\C^{d})$, where $d=\sum_v\deg(v) = 2|\E|$.

More generally, a \emph{global} self-adjoint boundary condition for $H$ is given by 
an arbitrary unitary operator $U$ on $\C^{d}$, as follows: 
\[ i(U-1)\fvec + (U+1)\nabla\fvec = 0, \] 
where $\fvec=\oplus_v\fvec(v)\in\C^{d}$, $\nabla\fvec=\oplus_v\nabla\fvec(v)\in\C^{d}$ 
are vectors of boundary values of $f_e$ and outgoing derivatives of $f_e$ at the ends of all edges $e\in\E$.
The matrix $U$ is also known as the bond scattering matrix $\sigma(-1)$, see \cite[sec.~2.1.2]{BerKuc_graphs}.
We use global boundary conditions in our proof of Theorem \ref{thm:sf-wind-loop}

\subsection*{Winding number.} 

Let $U_t\in\U(\C^d)$, $t\in S^1$ be a loop of unitary operators.
Its \emph{winding number} is defined as the degree of the composition
\[ S^1 \tor{U} \U(\C^d) \tor{\det} \U(\C)\cong S^1 \]
and is denoted $\wind(U_t)$.
Here $\det\colon\U(\C^d)\to\U(\C)$ is the determinant map
and $\U(\C) = \sett{e^{i\alpha}}{\alpha\in\R}$ is oriented in the direction of increasing of $\alpha$.
The winding number is a homotopy invariant of a loop and provides an isomorphism
\begin{equation}\label{eq:wind-pi1}
	\wind\colon\pi_1(\U(\C^d))\to\Z,  
\end{equation}
so two loops $S^1\to\U(\C^d)$ with the same winding number are homotopic.

\subsection*{Spectral flow formula.}

A continuous family $U_t$ of unitary operators on $\C^{d}$ 
leads to a continuous family of (Lagrangian) subspaces\footnote{The winding number of $U_t$ 
coincides with the Maslov index of the corresponding loop $L_t$.} 
in $\C^{d}\oplus\C^{d}$, so the corresponding family $H_t$ of realizations of $H$ 
is graph continuous by Proposition \ref{prop:HL-cont}.

The following theorem is the main result of the appendix.
It could be deduced from \cite[Theorem 4.1]{ivanov2023boundary} 
or \cite[Theorem 3]{prokhorova2024family}. 
However, it is instructive to give a proof for quantum graphs in their own language. 


\renewcommand\thethm{A}

\begin{thm}\label{thm:sf-wind-loop}
  Let $H$ be the Laplacian on a finite metric graph $\Gamma$, not necessarily connected.
	Let $U_t$, $t\in S^1$ be a loop of unitary operators in $\C^d$
	and $H_t$ be the corresponding loop of self-adjoint realizations of $H$.
	Then the spectral flow of $H_t$ is equal to the winding number of $U_t$:	
	\begin{equation}\label{eq:sf-wind}
		\spf(H_t) = \wind(U_t). 
	\end{equation}
	For local boundary conditions $U_t=\oplus_v U^v_t$, this can be written as
	\begin{equation}\label{eq:sf-wind-v}
		\spf(H_t) = \sum_{v\in\V}\wind(U^v_t). 
	\end{equation}
\end{thm}

\proof
Since the determinant is multiplicative with respect to the direct sum of operators,
$\wind(U_t) = \sum_{v\in\V}\wind(U^v_t)$, so \eqref{eq:sf-wind-v} follows from \eqref{eq:sf-wind}.

1. Let us first check the equality \eqref{eq:sf-wind} for the graph consisting of one edge $e=[0,\ell]$ 
and taken with the following loop of boundary conditions:
\begin{equation}\label{eq:delta-e}
	f'(0)=t\cdot f(0), \quad f'(\ell)=0, 
\end{equation}
where $t$ runs the extended real line $\RR = \R\cup\set{\infty}\cong S^1$
and $t=\infty$ corresponds to the condition $f(0)=0$.
In other words, we take the $\delta(t)$ condition $\nabla f = tf$ at one vertex 
and the fixed condition $\nabla f = 0$ at the other vertex.
In the unitary picture, the boundary condition at the first vertex is given by the unitary 
$U_t=(i-t)/(i+t)$.
The corresponding loop $\RR\to\U(\C)$ makes one rotation in the positive direction, 
so the right hand side of \eqref{eq:sf-wind} is equal to $1$.
The spectrum of each unbounded operator in this family can be explicitly computed 
and it can be easily seen that 
zero eigenvalue exists only for the value $t=0$ and has multiplicity $1$.
The corresponding spectral curve is a monotonously increasing function of $t$, 
so it intersects zero spectral level at $t=0$ in the positive direction.
Therefore, the spectral flow of this loop is equal to $1 = \wind(U_t)$ 
and hence \eqref{eq:sf-wind} is satisfied.

2. Let now $\Gamma$ be an arbitrary graph 
with an arbitrary loop of boundary conditions.
If two loops in $\U(\C^{d})$ have the same winding number $m$, then they are homotopic, 
so by Proposition \ref{prop:HL-cont} the corresponding loops of Hamiltonians are also homotopic 
and their spectral flows coincide.
Hence we only need to check \eqref{eq:sf-wind} for \textit{one} loop in $\U(\C^{d})$ 
with the winding number $m$, for every $m\in\Z$.
For $m=0$ take a constant loop, then both the spectral flow and winding number vanish.
For $m<0$, reverse the direction of the loop; 
then the signs of both the spectral flow and the winding number are reversed.
Therefore, we can suppose that $m>0$.

Choose an edge $e\in\Gamma$, and let $\Gamma_e$ be the subgraph of $\Gamma$ 
obtained by removing the edge $e$.
Fix an arbitrary boundary condition at $\Gamma_e$ and 
the concatenation of $m$ loops of boundary conditions \eqref{eq:delta-e} at $e$.
Let $U_t$ and $H_t$ be the corresponding loops of unitary operators in $\U(\C^{d})$ 
and unbounded operators on $L^2(\Gamma)$. 

Since the boundary conditions on $e$ and $\Gamma_e$ are independent one from another, 
the loop $H_t$ is in fact a loop of operators on the disjoint union $e\sqcup\Gamma_e$, 
or equivalently the orthogonal sum of two loops of operators on $L^2(e)$ and $L^2(\Gamma_e)$.
The second loop is constant, so its spectral flow vanishes.
The first loop, on $L^2(e)$, 
is the concatenation of $m$ copies of the loop described in (1).
Both the spectral flow and winding number are additive under direct sums and under concatenation of loops,
so \[ \spf(H_t) = 0+m\cdot 1 = m = \wind(U_t), \]
and thus \eqref{eq:sf-wind} holds for this loop of boundary conditions.
This completes the proof of the theorem.
\endproof

\subsection*{Another proof of the first part of Theorem \ref{thm:sf-HBa-path}}

Let $U^v_t$ be the unitary operator which determines the vertex boundary condition for $\HBa(t)$ at $v$.
For $v\notin\B$ the loop $U^v_t$ is constant, so its winding number vanishes.
For $v\in\B$ the loop $U^v_t$ corresponds to the $\del(t)$ loop of boundary conditions at $v$.
We claim that its winding number is equal to $1$; 
then \eqref{eq:sf-wind-v} provides 
\[ \spf(H_t) = \sum_{v\in\B}1 = |\B|. \]
In order to prove the claim, note first that $\del(t)$ condition depends continuously on $(\alpha,t)$,
so the homotopy type of the loop $\RR\ni t\mapsto U^v_t$ does not depend on $\alpha$.
Hence it is enough to compute the winding number of $U^v_t$ in the simplest case, for $\alpha=0$.
In this case, we have the $\delta(t)$ boundary condition, which can be written as follows:
\begin{equation}\label{eq:AB'}
	A_t\matr{ f_+ \\ f_-} + B_t\matr{ \nabla f_+ \\ \nabla f_-} = 0, \txt{where}	
	A_t=\matr{ 1 & -1 \\ 0 & -t}, \; B_t=\matr{ 0 & 0 \\ 1 & 1}
\end{equation}
(here we denoted $f_\pm = f(\vpm)$ for brevity;
recall that $\nabla f_+ = f'_+$ and $\nabla f_- = -f'_-$).
The unitary operator $U^v_t$ is then defined by the formula $U^v_t = -(A_t-iB_t)\inv(A_t+iB_t)$,
see \cite[proof of thm.~1.4.4]{BerKuc_graphs}.
Substituting our values of $A_t$ and $B_t$, we obtain
 \[  U^v_t = -\matr{ 1 & -1 \\ i & -t-i}\inv\cdot\matr{ 1 & -1 \\ i & -t+i}
   = \frac{1}{t+2i}\matr{ -t & 2i \\ 2i & -t}, \]    
so $\det U^v_t = (t-2i)/(t+2i)$. 
When $t$ runs $\RR$ in the positive direction,
the complex number $(t-2i)/(t+2i)\in\U(\C)$ makes one revolution in the counter-clockwise direction.
Therefore, $\wind(U^v_t)=1$.
This proves our claim and thus the first part of Theorem \ref{thm:sf-HBa-path}.
\endproof

\begin{rem*}
	Of course, the operator $U^v_t$ and its determinant can be computed explicitly 
	for an arbitrary value of $\alpha$.
	But the homotopy invariance of the winding number spare us from superfluous computations.
\end{rem*}

\end{document}